\documentclass[10pt]{amsart}

\usepackage{amsmath,amssymb,amsfonts}
\usepackage{mathrsfs,latexsym,amsthm,enumerate}
\usepackage[all]{xy}

\usepackage{tikz}
\usepackage{verbatim}
\usepackage{graphicx}
\DeclareGraphicsExtensions{.png}

\newtheorem{theorem}{Theorem}[section]
\newtheorem{proposition}[theorem]{Proposition}
\newtheorem{example}[theorem]{Example}
\newtheorem{corollary}[theorem]{Corollary}

\newtheorem{lemma}[theorem]{Lemma}
\newtheorem{remark}[theorem]{Remark}

\input xypic

\title[Self-similar actions]{A correspondence between a class of monoids  and self-similar group actions II}

\author{M.~V.~Lawson}
\address{Department of Mathematics
and the
Maxwell Institute for Mathematical Sciences,
Heriot-Watt University,
Riccarton,
Edinburgh~EH14~4AS,
UK}
\email{markl@ma.hw.ac.uk}

\author{A.~R.~Wallis}
\address{Department of Mathematics
and the
Maxwell Institute for Mathematical Sciences,
Heriot-Watt University,
Riccarton,
Edinburgh~EH14~4AS,
UK}
\email{arw5@hw.ac.uk}

\dedicatory{Dedicated to the memory of David Rees}

\begin{document}
\begin{abstract} 
The first author showed in a previous paper that 
there is a correspondence between self-similar group actions and a class of left cancellative monoids called left Rees monoids. 
These monoids can be constructed either directly from the action using  Zappa-Sz\'ep products,
a construction that ultimately goes back to Perrot,
or as left cancellative tensor monoids from the covering bimodule, 
utilizing a construction due to Nekrashevych,
In this paper, we generalize the tensor monoid construction to arbitrary bimodules.
We call the monoids that arise in this way Levi monoids and show that they are precisely the equidivisible monoids equipped with length functions.
Left Rees monoids are then just the left cancellative Levi monoids.
We single out the class of irreducible Levi monoids and prove that they are determined by an isomorphism
between two divisors of its group of units.
The irreducible Rees monoids are thereby shown to be determined by a partial automorphism of their group of units;
this result turns out to be signficant since it connects irreducible Rees monoids directly with HNN extensions.
In fact, the universal group of an irreducible Rees monoid is an HNN extension of the group of units by a single stable letter
and every such HNN extension arises in this way.\\

\end{abstract}

\maketitle

\thanks{The first author's research was partially supported by an EPSRC grant (EP/I033203/1)
and the second by an EPSRC Doctoral Training Account (EP/P504945/1). 
Some of the material in this paper appeared in the second author's PhD thesis \cite{AW}.}

\section{Introduction}\setcounter{theorem}{0}

The theory of self-similar groups has two aspects: a group-theoretical and a monoid-theoretical.
The group-theoretical is well-known, being the subject of the 2005 book by Nekrashevych \cite{N1},
and originating in the 1980's.
The monoid-theoretical is less well-known.
The first author showed \cite{Lawson2008a} that self-similar groups were also defined 
in the 1972 thesis of J.-F.~Perrot \cite{P1} (see also \cite{P2}).
They arose as part of a generalization of the theory of polycyclic inverse monoids of \cite{NP}
along the lines of David Rees's 
pioneering paper \cite{Rees}.\footnote{Perrot's theory 
can also be viewed as a special case of a later one developed by McAlister \cite{McAlister74}.}
The starting point for this paper are correspondences, established in \cite{Lawson2008a}, between three classes of mathematical structures:
\begin{enumerate}
\item Self-similar group actions defined in full generality without the assumption that the action be faithful.
\item A class of left cancellative monoids, called left Rees monoids.
\item A class of $0$-bisimple inverse monoids with zero.
\end{enumerate}
This paper will deal with the ramifications and generalizations of these characterizations focusing on the first two.

We shall rely on basic notions from semigroup theory.
Our references for this are \cite{Howie,Lallement2}.
In particular, we recall the definitions of Green's relations in a monoid $S$.
Define $a \, \mathscr{L} \, b$ if and only if $Sa = Sb$,
$a \, \mathscr{R} \, b$ if and only if $aS = bS$
and 
$a \, \mathscr{J} \, b$ if and only if $SaS = SbS$.
The relation $\mathscr{H} = \mathscr{L} \cap \mathscr{R}$.
It can be proved that $\mathscr{L} \circ \mathscr{R} = \mathscr{R} \circ \mathscr{L}$;
define $\mathscr{D} = \mathscr{L} \circ \mathscr{R}$.
If $\mathscr{K}$ is one of Green's relations then $K_{a}$ is the $\mathscr{K}$-class containing $a$.
A semigroup is said to be {\em right stable} if $a \, \mathscr{J} \, ab$ implies that $a \, \mathscr{R} \, ab$.
We define {\em left stability} dually, and a semigroup that is both left and right stable is said to be {\em stable}.
See \cite{Lallement2} for more on this notion.\\

\noindent
{\bf Acknowledgements } This paper was inspired by a few pages in the first edition of Cohn's book \cite{Cohn}.
There he develops a theory of group embeddings of a class of cancellative monoids he calls `rigid'.
The theory reminded us of Bass-Serre theory but defined for cancellative monoids rather than groups and led to the work in Sections~4 and 5.
We would also like to thank Alan Cain and Stuart Margolis for some useful discussions.

As we were putting the finishing touches to this paper, we learnt of the sad news that Prof David Rees FRS had died on 16th August 2013.
Although he only wrote a few papers on semigroup theory, his main contributions were in commutative algebra,
they have proved extremely influential.
The ideas developed in this paper are just one example of that influence.

\section{Levi monoids}

A {\em length function} on a monoid $S$ is a homomorphism $\lambda \colon S \rightarrow \mathbb{N}$ to the additive monoid
of natural numbers such that $\lambda^{-1}(0)$ is the group of units of $S$.

\begin{lemma}\label{le: Greens_relations} 
Let $S$ be a monoid with group of units $G$ equipped with a length-function $\lambda$.
\begin{enumerate}

\item If $aS \subseteq bS$ or $Sa \subseteq Sb$ or  $SaS \subseteq SbS$ then $\lambda (a) \geq \lambda (b)$.

\item $a \, \mathscr{L} \, b \Leftrightarrow Ga = Gb$.

\item $a \, \mathscr{R} \, b \Leftrightarrow aG = bG$.

\item $a \, \mathscr{J} \, b \Leftrightarrow GaG = GbG$.

\item $\mathscr{D} = \mathscr{J}$.

\item $S$ is stable.

\end{enumerate}
\end{lemma}
\begin{proof} (1) By assumption, $a = bs$ for some $s$.
Thus $\lambda (a) = \lambda (b) + \lambda (s)$ and so $\lambda (a) \geq \lambda (b)$.
The other cases are proved similarly.

(2) Suppose that  $a \, \mathscr{L} \, b$.
Then $a = xb$ and $b = ya$.
In particular, $a = xb = xya$.
Thus $\lambda (a) = \lambda (x) + \lambda (y) + \lambda (a)$.
It follows that $\lambda (x) = \lambda (y) = 0$ and so both $x$ and $y$ are invertible.
The proof of the converse is immediate.

The proofs of (3) and (4) are similar to the proof of (2).

The proof of (5) follows immediately from (2), (3) and (4).

To prove (6), suppose that $a \, \mathscr{J} \, ab$. 
Then $a = gabh$ for some $g,h \in G$ by (4) above.
It follows that $\lambda (a) = \lambda (a) + \lambda (b)$.
Thus $\lambda (b) = 0$ and so $b$ is invertible.
By (3) above, we have that $a \, \mathscr{R} \, ab$;
the dual case follows by a similar argument.
\end{proof}

Let $S$ be a monoid.
It  is said to be {\em equidivisible} if for all $a,b,c,d \in S$
we have that $ab = cd$ implies $a = cu$ and $d = ub$ for some $u \in S$ or
$c = av$ and $b = vd$ for some $v \in S$.
See \cite{MS} for more information.
An {\em atom} in $S$ is a noninvertible element $a$ such that if $a = bc$ then either $b$ or $c$ is invertible.

We define a {\em Levi monoid} to be an equidivisible monoid equipped with a length function which contains at least one noninvertible element
where the last condition is there simply to exclude groups.

\begin{example}{\em  The most natural examples of equidivisible monoids with length functions are the free monoids.
A {\em free monoid} $X^{\ast}$ on a non-empty set $X$, called the set of {\em letters}, 
consists of all finite sequences of elements of $X$ called {\em strings}, 
including the empty string $\varepsilon$, which we often denote by 1, 
with multiplication given by {\em concatenation}
of strings. 
The {\em length} $|x|$ of a string $x$ is the total number of letters that occur in it.
We denote the set of all strings of length $n$ by $X^{n}$.
If $x = yz$ then $y$ is called a {\em prefix} of $x$.
We write $x \preceq y$ in this case. 
The relation  $\preceq$ is called the {\em prefix ordering}.}
\end{example}

The following is proved in \cite{Lallement2} and attributed to F.~W.~Levi and is the source of our terminology.

\begin{theorem} 
A monoid $S$ is free if and only if it is a Levi monoid with a trivial group of units.
\end{theorem}

\subsection{Normalized length functions}

We begin with a simple result.

\begin{lemma} Let $S$ be a Levi monoid.
If $\lambda (a) = 1$ then $a$ is an atom.
\end{lemma}
\begin{proof} Suppose that $a = bc$.
Then $1 = \lambda (b) + \lambda (c)$.
It follows that $\lambda (b) = 0$ or $\lambda (c) = 0$ and so either $b$ or $c$ is invertible.
\end{proof}

The converse to the above result is not, however, true in general.

\begin{example} 
{\em Let $S = (a + b)^{\ast}$ be the free monoid on two generators.
Let $\lambda (a) = m$ and $\lambda (b) = n$ where $m$  and $n$ are any fixed non-zero natural numbers.
Then we may extend $\lambda$ to the whole of $S$ to obtain a length function.
This length function will have the property that $\lambda^{-1}(1)$ only consists of atoms when $m = 1 = n$ 
which corresponds to the case where $\lambda$ is the usual length function on the free monoid.}
\end{example}

A length function $\lambda$ on a Levi monoid that satisfies the additional condition that 
$\lambda (a) = 1$ if and only if $a$ is an atom is said to be {\em normalized}.
In this section, we shall prove that every Levi monoid has exactly one normalized length function.

\begin{lemma}\label{le: bradbury} Let $S$ be a monoid with group of units $G$ equipped with a length function.
\begin{enumerate}

\item $aS = S$ if and only if $a$ is invertible.

\item The element $a$ is an atom if and only if $aS$ is a maximal proper principal right ideal if and only if $Sa$
is a maximal proper principal left ideal.

\item Each noninvertible element of $S$ is contained in a maximal proper principal right ideal.

\item Each noninvertible element of $S$ can be written as a product of a finite number of atoms.

\item Let $X$ be a transversal of the generators of the maximal proper principal right ideals.
Then $S = \langle X \rangle G$ where $\langle X \rangle$ is the submonoid generated by $X$.
 
\end{enumerate}
\end{lemma}
\begin{proof} (1) If $a$ is invertible then $aS = S$.
Conversely, suppose that $aS = S$. Then $1 = as$ for some $s$.
Hence $\lambda (as) = 0$.
It follows that $\lambda (a) = 0$ and so $a$ is invertible.

(2) We prove the claim for principal right ideals; the other claim follows by symmetry.
Let $aS$ be a maximal proper principal right ideal.
Suppose that $a = bc$.
Then $aS \subseteq bS$.
By assumption either $aS = bS$ or $bS = S$.
By Lemma~\ref{le: Greens_relations},
the former implies that $a = bg$ for some invertible element $g$.
Thus $bg = bc$.
But then $0 = \lambda (g) = \lambda (c)$ and so $c$ is invertible.
By (1) above,
the latter implies that $b$ is invertible.
We have therefore proved that $a$ is an atom.
Conversely, let $a$ be an atom.
Suppose that $aS \subseteq bS$.
Then $a = bs$ for some $s$.
If $b$ is invertible then $bS = S$ by (1) above.
If $s$ is invertible then $aS = bS$.
We have therefore proved that $aS$ is a maximal proper principal right ideal.

(3) Let $s \in S$ be any noninvertible element.
If $sS$ is a maximal proper principal right ideal then we are done by (2) above.
If not then $sS \subset a_{1}S$ for some $a_{1} \in S$.
We may write $s = a_{1}b$.
If  $\lambda (s) = \lambda (a_{1})$ then $b$ would be invertible and we would have $sS = a_{1}S$.
It follows that $\lambda (s) > \lambda (a_{1})$.
If $a_{1}S$ is a maximal proper principal right ideal then we are done,
or the process continues with $a_{1}$ instead of $a$.
But the length function now tells us that this process must conclude in a finite number of steps,
and this can only happen when we reach a maximal proper principal right ideal containing $s$.

(4) Let $s \in S$ be any noninvertible element.
By (2) and (3), we may write $s = a_{1}s_{1}$ where $a_{1}$ is an atom and $\lambda (s) > \lambda (s_{1})$.
The process may now be repeated, the length function guaranteeing that the process terminates.

(5) Let $s = a_{1} \ldots a_{n}$ be a representation of $s$ as a product of atoms.
For each atom $a$ there exists $x \in X$ and $g \in G$ such that $a = xg$.
Therefore we may write $a_{i} = x_{i}g_{i}$ for some $x_{i} \in X$ and $g_{i} \in G$.
Thus $s = (x_{1}g_{1}) \ldots (x_{n}g_{n})$.
We now come to the key observation.
Let $g \in G$ and $x \in X$ be arbitrary.
Then $Sgx = Sx$ is a maximal proper principal left ideal and so $gx$ is an atom.
It follows that $gx = x'g'$ for some $x' \in X$ and $g' \in G$ by assumption.
Applying this representation in turn from left to right,
we may write $s = y_{1} \ldots y_{n}g$ for some $y_{i} \in X$ and $g \in G$.
\end{proof}

Let $aS \subseteq bS$.
Denote by $[aS,bS]$ the set of all principal right ideals $cS$ such that $aS \subseteq cS \subseteq bS$.

\begin{lemma}\label{le: right_ideals} Let $S$ be a Levi monoid.
\begin{enumerate}

\item The set  $[aS,bS]$  is linearly ordered and finite.

\item $a[bS,S] = [abS,aS]$.

\item $[abS,S] = a[bS,S] \cup [aS,S]$.

\item If $bS \subseteq cS \subseteq S$ and $bS \subseteq dS \subseteq S$
are such that $acS = adS$ then $cS = dS$.

\item $\lambda (a) \geq \left| [aS,S]  \right| - 1$ for any $a \in S$.

\end{enumerate}
\end{lemma}
\begin{proof} (1) Let $aS \subseteq xS \subseteq bS$ and let $aS \subseteq yS \subseteq bS$.
By assumption, $a \in xS \cap yS$. 
Thus $a = xu = yv$ for some $u,v \in S$.
By equidivisibility, there exists either a $w$ such that
$x = yw$ and $v = wu$ or a $z$ such that $y = xz$ and $u = zv$.
If the former, then $xS \subseteq yS$, and if the latter then $yS \subseteq xS$.
Thus the set  $[aS,bS]$ is linearly ordered. 

Let $aS \subseteq cS \subset dS \subseteq bS$.
Then $\lambda (c) \geq \lambda (d)$.
Let5 $c = ds$ for some $s \in S$.
If $\lambda (c) = \lambda (d)$ then $s$ is invertible and $cS = dS$.
Hence  $\lambda (c) > \lambda (d)$.
It is now immediate that  $[aS,bS]$ is  finite.

(2) Clearly, the lefthand side is contained in the righthand side.
We now show that the righthand side is contained in the lefthand side.
Let $abS \subseteq cS \subseteq aS$.
Then $ab = cx$ and $c = ay$.
Suppose first that $y$ is invertible.
Then $cS = aS$ and we are done.
In what follows, therefore, we shall assume that $y$ is not invertible.
By equidivisibility, the equation $ab = cx$ implies that either there is an element $u$ such that
$a = cu$ and $x = ub$ or an element $v$ such that $c = av$ and $b = vx$.
Suppose first that the former occurs.
Then $a = cu = ayu$.
Using the length function, we deduce that $\lambda (yu) = 0$ and so $y$ is invertible.
This contradicts our assumption.
It follows that  $c = av$ and $b = vx$.
Observe that $vS \in [bS,S]$ since $b = vx$.
But $a(vS) = avS = cS$, as required.

(3) Observe first that $abS \subseteq aS \subseteq S$.
Let $cS$ be an element of the lefthand side.
Since $[abS,S]$ is a linearly ordered set by (1) above, 
we have that $abS \subseteq cS \subseteq aS$ or $aS \subseteq cS \subseteq S$.
But by (1) above, in either case, we have that $cS$ belongs to the righthand side.
The reverse inclusion is immediate.

(4) We have that $cS$ and $dS$ are comparable.
Without loss of generality, we may assume that $cS \subseteq dS$.
But from $acS = adS$ we deduce that $\lambda (c) = \lambda (d)$.
It follows that $cS = dS$, as claimed.

(5) Suppose that  $\left| [aS,S] \right| = 1.$
This of course means that $aS = S$ which occurs if and only if $a$ is invertible by Lemma~\ref{le: bradbury}.
Thus the inequality holds when $a$ is invertible.

Suppose that  $\left| [aS,S] \right| = 2.$
Then $aS \subseteq S$ and there are no principal right ideals inbetween.
This happens if and only if $a$ is an atom by part (1) of Lemma~\ref{le: bradbury}.
By definition, $\lambda (a) \geq 1$ and so the inequality holds again.

Suppose that  $\left| [aS,S] \right|  = n \geq 3.$
Then $aS = a_{1}S \subset a_{2}S \subset \ldots \subset a_{n}S = S$.
We have that $a = a_{1} = a_{2}x_{2}$ where $x_{2}$ is not invertible;
$a_{2} = a_{3}x_{3}$ where $x_{3}$ is not invertible and so on.
We deduce that $a = a_{n}x_{n}x_{n-1} \ldots x_{2}$ where the $x_{i}$ are not invertible and $a_{n}$ is invertible.
It follows that $\lambda (a) = \lambda (x_{2}) + \ldots + \lambda (x_{n})$.
Thus $\lambda (a) \geq n-1$, as claimed.
\end{proof}

The above lemma provides all we need to prove the main result of this section.

\begin{proposition}\label{prop: normalized} Let $S$ be a Levi monoid.
\begin{enumerate}

\item Define the function $\nu \colon S \rightarrow \mathbb{N}$ by 
$$\nu (a) = \left|  [aS,S] \right| - 1$$
for all $a \in S$.
Then $\nu$ is a normalized length function on $S$.

\item The function $\nu$  is the only normalized length function on $S$.

\end{enumerate}
\end{proposition}
\begin{proof} (1)  It follows by Lemma~\ref{le: right_ideals} that $\nu$ is a homomorphism.
We have that $\nu (a) = 0$ if and only if $aS = S$ if and only if $a$ is invertible.
We have that $\nu (a) = 1$ if and only if $aS$ is a maximal proper principal right ideal if and only if $a$ is an atom. 

(2) This follows from part (4) of Lemma~\ref{le: bradbury}.
\end{proof}

\begin{remark}
{\em We shall assume from now on that the length function we work with on a Levi monoid is normalized.}
\end{remark}

\begin{remark}
{\em In Example~1.8 of Chapter~5 of \cite{Lallement2}, Lallement constructs a cancellative equidivisible monoid with a trivial group of units
that does not have a length function in our sense. 
This is an indication that the class of Levi monoids might be worth generalizing perhaps by using different kinds of length functions.}
\end{remark}

We have seen that every noninvertible element in a Levi monoid can be written as a product of atoms.
The following result describes what kind of uniqueness we can expect in such a product and provides a clue to the underlying structure of Levi monoids
that we shall pursue in the next section.

\begin{lemma}\label{le: uniqueness} Let $S$ be a Levi monoid.
 Suppose that 
$$x = a_{1} \ldots a_{m} = b_{1} \ldots b_{n}$$ 
where the $a_{i}$ and $b_{j}$ are atoms.

\begin{enumerate}

\item $m = n$.

\item There are invertible elements $g_{1}, \ldots, g_{n-1}$ such that
$$a_{1} = b_{1}g_{1}, \quad
a_{2} = g_{1}^{-1}b_{2}g_{2}, \quad
\ldots
\quad
a_{n} = g_{n-1}^{-1}b_{n}.$$ 

\end{enumerate}
\end{lemma}
\begin{proof} (1). This is immediate from the properties of normalized length functions.

(2). We bracket as follows
$$a_{1}(a_{2}  \ldots a_{m}) = b_{1}(b_{1} \ldots b_{m}).$$
By equidivisibility, 
$a_{1} = b_{1}u$ and $b_{2} \ldots b_{m} = u a_{2} \ldots a_{m}$ for some $u$ 
or 
$b_{1} = a_{1}v$ and $a_{2} \ldots a_{m} = vb_{2} \ldots b_{m}$ for some $v$.
In either case, $u$ and $v$ are invertible since both $a_{1}$ and $b_{1}$ are atoms using the length function.
Thus 
$a_{1} = b_{1}g_{1}$, where $v = g_{1}$ is invertible,
and 
$b_{2} \ldots b_{m} = g_{1}a_{2} \ldots a_{m}$.

We now repeat this procedure bracketing thus
$$b_{2}(b_{3} \ldots b_{m}) = g_{1}a_{2}(a_{3} \ldots a_{m}).$$
By the same argument as above, we get that
$g_{1}a_{2} = b_{2}g_{2}$ for some invertible element $g_{2}$
and 
$b_{3} \ldots b_{m} = g_{2}a_{3} \ldots a_{m}$.

The process continues and we obtain the result.
\end{proof}

Result (2) above may usefully be presented by means of the following {\em interleaving diagram}.
$$\spreaddiagramrows{1pc}
\spreaddiagramcolumns{1pc}
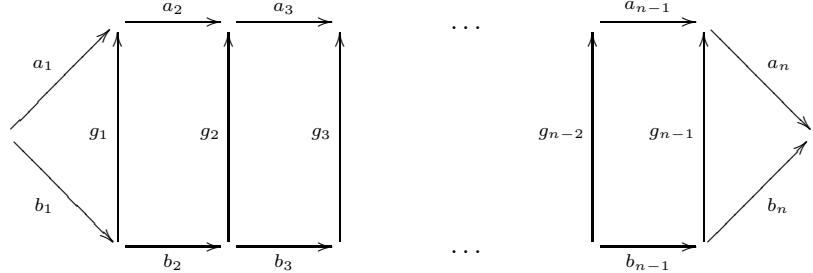
\diagram
& 
& \rto^{a_{2}} 
& \rto^{a_{3}}
&
& \ldots
& \rto^{a_{n-1}}
& \drto^{a_{n}}
&
\\
& \urto^{a_{1}} \drto_{b_{1}}  
& 
&
&
&
&
&
&
\\
& 
& \uuto^{g_{1}} \rto_{b_{2}} 
& \uuto^{g_{2}} \rto_{b_{3}}
& \uuto^{g_{3}} 
& \ldots
& \uuto^{g_{n-2}} \rto_{b_{n-1}}
& \uuto^{g_{n-1}} \urto_{b_{n}}
&
\enddiagram$$

\subsection{Tensor monoids}

In this section, we shall describe a procedure for constructing all Levi monoids that generalizes a construction to be found in \cite{N1}.
The proof of the following is immediate from the properties of normalized length functions.

\begin{lemma} Let $S$ be a Levi monoid.
Let $G$ be its group of units and $X$ the set of atoms.
Then for all $g,h \in G$ and $x \in X$ we have that $gxh \in X$.
\end{lemma}

In the light of the lemma above, we make the following definitions.
Let $X$ be any non-empty set and $G$ a group such that $G$ acts on $X$ both on the left and the right
in such a way that the two actions are conformable meaning that $(gx)h = g(xh)$ for all $g,h \in G$ and $x \in X$.
Then we say that $X$ is a {\em $(G,G)$-bimodule} or  {\em $G$-bimodule} or just a {\em bimodule (over the group $G$)}.

It follows that with every Levi monoid, we may associate a bimodule over its group of units called its {\em associated bimodule (of atoms)}.

Let $U$ be a $G$-bimodule and $V$ an $H$-bimodule.
A {\em morphism} from $U$ to $V$ is a pair $(\alpha,\beta)$ where $\alpha \colon G \rightarrow H$ is
a group homomorphism and $\beta \colon \: U \rightarrow V$ is a function such that
$\beta (g_{1}ug_{2}) = \alpha (g_{1})\beta (u) \alpha (g_{2})$ for all $g_{1},g_{2} \in G$ and $u \in U$.
If $U$ and $V$ are both bimodules over the same group $G$ then we shall usually require that $\alpha$ is the identity homomorphism.
{\em Isomorphisms}  between bimodules are defined in the obvious way.
A monoid homomorphism between Levi monoids is called {\em atom preserving} if it maps atoms to atoms.
The proof of the following is now immediate.

\begin{lemma}\label{le: functor1} The construction of the associated bimodule is a functor from the category of Levi monoids
and atom preserving monoid homomorphisms to the category of bimodules and their morphisms.
\end{lemma}

\begin{remark}
{\em Let $S$ be a monoid with group of units $G$.
Then $S$ itself becomes a $G$-bimodule under left and right multiplication.
We shall use this construction below.}
\end{remark}

Our goal now is to show that from each bimodule we may construct a Levi monoid.
Our tool for this will be tensor products and the construction of a suitable tensor algebra analogous to the tensor algebra of module theory; 
see Chapter~6 of \cite{Street}, for example.
We recall the key definitions and results we need first.

Let $G$ be a group that acts on the set $X$ on the right and the set $Y$ on the left.
A function $\alpha \colon X \times Y \rightarrow Z$ to a set $Z$ is called a {\em bi-map} or a {\em 2-map}
if $\alpha (xg,y) = (x,gy)$ for all $(xg,y) \in X \times Y$ where $g \in G$.
We may construct a universal such bi-map $\lambda \colon X \times Y \rightarrow X \otimes Y$
in the usual way \cite{Howie}.
However, there is a simplification in the theory due to the fact that we are acting by means of a group.
The element $x \otimes y$ in $X \otimes Y$ is the equivalence class of $(x,y) \in X \times Y$ under the relation $\sim$
where $(x,y) \sim (x',y')$ if and only if $(x',y') = (xg^{-1},gy)$ for some $g \in G$.

Suppose now that $X$ is a $G$-bimodule.
We may therefore define the tensor product $X \otimes X$ as a set.
We may also define $g(x \otimes y) = gx \otimes y$ and $(x \otimes y)g = x \otimes yg$.
Observe that $x \otimes y = x' \otimes y'$ implies that $gx \otimes y = gx' \otimes y'$, and dually.
It follows that $X \otimes X$ is a also a bimodule.
Put $X^{\otimes 2} = X \otimes X$.
More generally, we may define  $X^{\otimes n}$ for all $n \geq 1$ using $n$-maps, and we define $X^{\otimes 0} = G$
where $G$ acts on itself by multiplication on the left and right.
The proof of the following lemma is almost immediate from the definition and the
fact that we are acting by a group.
It should be compared with Lemma~\ref{le: uniqueness}.

\begin{lemma}\label{le: equality_of_tensors}
Let $n \geq 2$.
Then
$$x_{1} \otimes \ldots \otimes x_{n} = y_{1} \otimes \ldots \otimes y_{n}$$
if and only if 
there are elements $g_{1}, \ldots, g_{n-1} \in G$
such that
$y_{1} = x_{1}g_{1}$,
$y_{2} = g_{1}^{-1}x_{2}g_{2}$,
$y_{3} = g_{2}^{-1}x_{3}g_{3}$,
$\ldots$,
$y_{n} = g_{n-1}^{-1}x_{n}$.
\end{lemma}

Let $X$ be a $G$-bimodule.
Define
$$\mathsf{T}(X) = \bigcup_{n = 0}^{\infty} X^{\otimes n}.$$
We shall call this the {\em tensor monoid} associated with the $G$-bimodule $X$.
Observe that we may regard $X$ as a subset of $\mathsf{T}(X)$;
we denote the inclusion map by $\iota$. 
The justification for our terminology will follow from (1) below.

\begin{theorem}\label{the: first_theorem}\mbox{}
\begin{enumerate}

\item Let $X$ be a $G$-bimodule. 
Then $\mathsf{T}(X)$ is a Levi monoid with group of units $G$ whose associated $G$-bimodule is $X$.

\item  Let $X$ be a $G$-bimodule.
Let $S$ be any monoid with group of units $H$.
Let $(\alpha,\beta)$ be a bimodule morphism from $X$ to $S$ where $\alpha \colon G \rightarrow H$
and $\beta \colon X \rightarrow S$.
Then there is a unique monoid homomorphism $\theta \colon \mathsf{T}(X) \rightarrow S$
such that $\theta \iota = \beta$ and which which agrees with $\alpha$ on $G$ and $\beta$ on $X$.

\item Every Levi monoid is isomorphic to the tensor monoid of its associated bimodule.

\end{enumerate}
\end{theorem}
\begin{proof} (1)  Multiplication is tensoring of sequences and left and right actions by elements of $G$.
We define $\lambda (g) = 0$ where $g \in G$ and $\lambda (x_{1} \otimes \ldots \otimes x_{n}) = n$.
Formally, we are using the fact that there is a canonical isomorphism 
$$X^{\otimes p} \otimes X^{\otimes q} \cong X^{\otimes (p+q)}.$$

Let $\mathbf{a},\mathbf{b},\mathbf{c},\mathbf{d} \in \mathsf{T}(X)$
and suppose that $\mathbf{a} \otimes \mathbf{b} = \mathbf{c} \otimes \mathbf{d}$.
Let $\mathbf{a} = a_{1} \otimes \ldots \otimes a_{m}$,
$\mathbf{b} = b_{m+1} \otimes \ldots \otimes b_{r}$,
$\mathbf{c} = c_{1} \otimes \ldots \otimes c_{n}$
and
$\mathbf{d} = d_{n+1} \otimes \ldots \otimes d_{r}$.
Without loss of generality, we assume that $m < n$ and put $s = n - m$.
We now apply Lemma~\ref{le: equality_of_tensors}.
Let the invertible elements involved be $g_{1}, \ldots, g_{r-1}$.
Define $\mathbf{v} = b_{m+1} \otimes \ldots \otimes b_{m+s}g_{m+s}^{-1}$.
Then $\mathbf{c} = \mathbf{a} \otimes \mathbf{v}$ and $\mathbf{b} = \mathbf{v} \otimes \mathbf{d}$.
Thus the tensor monoid is equidivisible.

The elements of length 0 are precisely the elements of $G$ and so the invertible elements;
the elements of length 1 are precisely the elements of $X$.
It follows that $\mathsf{T}(X)$ is a Levi monoid and that $\lambda$ is its normalized length function.

(2) Define $\theta \colon \mathsf{T}(X) \rightarrow S$ as follows.
If $g \in G$ then $\theta (g) = \alpha (g)$.
Otherwise $\theta(x_{1} \otimes \ldots  \otimes x_{m}) = \beta (x_{1}) \ldots \beta (x_{m})$.
We need to show that this is well-defined and for this we shall use Lemma~\ref{le: equality_of_tensors}.
Suppose that 
$$x_{1} \otimes \ldots \otimes x_{m} = y_{1} \otimes \ldots \otimes y_{m}.$$
Then there are elements $g_{1}, \ldots, g_{m-1} \in G$
such that
$y_{1} = x_{1}g_{1}$,
$y_{2} = g_{1}^{-1}x_{2}g_{2}$,
$y_{3} = g_{2}^{-1}x_{3}g_{3}$,
$\ldots$
$y_{m-1} = g_{m-2}^{-1}x_{m-1}g_{m-1}$,
$y_{m} = g_{m-1}^{-1}x_{m}$.
We have that
$$\beta (y_{1}) \ldots \beta (y_{m})
=
\beta(x_{1}g_{1}) \beta (g_{1}^{-1}x_{2}g_{2}) \ldots \beta (g_{m-2}^{-1}x_{m-1}g_{m-1})\beta (g_{m-1}^{-1}x_{m}).$$
But now we use the fact that $\beta$ is part of a morphism.
Thus $\beta (x_{1}g_{1}) = \beta (x_{1})\alpha (g_{1})$ and so on.
The terms in $\alpha$ cancel out and we are left with $\beta (x_{1}) \ldots \beta (x_{m})$.
It follows that
$$\theta (x_{1} \otimes \ldots  \otimes x_{m})
=
\theta (y_{1} \otimes \ldots \otimes y_{m}),$$
as required.
By construction $\theta$ is a monoid map and agrees with $\alpha$ and $\beta$ as indicated.
Uniqueness follows from the fact that $G$ and $X$ together constitute a generating set for the
tensor monoid.

(3) Let $S$ be a Levi monoid with group of units $G$ and set of atoms $X$.
Then by (2) above there is a monoid homomorphism $\Theta \colon \mathsf{T}(X) \rightarrow S$ 
which induces a bijection between the group of units of  $\mathsf{T}(X)$ and the group of units of $G$
and between the atoms of  $\mathsf{T}(X)$ and the atoms of $S$.
By part (4) of Lemma~\ref{le: bradbury}, this homomorphism is surjective.
By Lemma~\ref{le: uniqueness} and Lemma~\ref{le: equality_of_tensors},
this homomorphism is injective.
\end{proof}

\begin{example} 
{\em If we let $X$ be any non-empty set and let $G$ be the trivial group
then the tensor monoid constructed from the trivial bimodule that arises is nothing other than the free monoid on $X$.}
\end{example}

\begin{lemma}\label{le: functor2} The construction of the tensor monoid is a functor from the category of bimodules and their morphisms 
to the category of Levi monoids and atom preserving monoid homomorphisms.
\end{lemma}
\begin{proof} Let $(\alpha,\beta)$ be a morphism from the $G$-bimodule $X$ to the $H$-bimodule $Y$.
The morphism $(\alpha,\beta)$ can equally well be regarded as a morphism from $X$ to $\mathsf{T}(Y)$.
By part (2) of Theorem~\ref{the: first_theorem}, this extends to a monoid homomorphism
$\theta \colon \mathsf{T}(X) \rightarrow \mathsf{T}(Y)$ that is atom preserving.
\end{proof}

From Lemma~\ref{le: functor1} and Lemma~\ref{le: functor2} together with part (2) of Theorem~\ref{the: first_theorem},
we have in fact proved the following.

\begin{theorem} 
The category of Levi monoids and atom preserving monoid homomorphisms 
and the category of bimodules and their morphisms are equivalent.
\end{theorem}

We now look at some important special cases of our construction.
Let $X$ be a $G$-bimodule.
If whenever $xg = x$ we have that $g$ is an identity, 
then we say the action is {\em right free}.
A bimodule which is right free is called a {\em covering bimodule}; 
this terminology is taken from \cite{N1}.
We define {\em left free} dually.
A bimodule which is both left and right free is said to be {\em bifree}.

\begin{lemma} Let $S$ be a Levi monoid with group of units $G$ and set of atoms $X$.
\begin{enumerate}

\item The $G$-bimodule $X$ is right free if and only if $S$ is left cancellative.

\item The $G$-bimodule $X$ is left free if and only if $S$ is right cancellative.

\item The $G$-bimodule $X$ is bifree if and only if $S$ is  cancellative.

\end{enumerate}
\end{lemma}
\begin{proof} Clearly, we need only prove (1).
It is immediate that if $S$ is left cancellative then $X$ is right free.
We prove the converse.
Suppose that $X$ is right free.
We prove that $S$ is left cancellative.
Let $ab = ac$.
Then by equidivisibility and the fact that we have a length function there is, without loss of generality,
an invertible element $u$ such that $a = au$ and $c = ub$.
Our claim will be proved if we can show that $a = au$ implies that $u$ is the identity.
If $a$ is an atom then this is immediate by our assumption that $X$ is right free.
More generally,  by part (4) of Lemma~\ref{le: bradbury}we may write $a = a_{1} \ldots a_{m}$ where the $a_{i}$ are atoms.
Thus $ a_{1} \ldots a_{m} =  a_{1} \ldots a_{m}u$.
We now use part (2) of Lemma ~\ref{le: uniqueness} to deduce that $u$ is the identity, as required.
\end{proof}

Left cancellative Levi monoids are called {\em left Rees monoids} and cancellative Levi monoids are called {\em Rees monoids}.
We see that left Rees monoids are constructed from covering bimodules.

\begin{remark}
{\em The tensor monoid of a covering bimodule, without any further properties, is constructed in \cite{N2}.
The discussion there of the Fock tree and its symbolic labelling is `really' about the structure of the $\mathscr{R}$-classes of the monoid.}
\end{remark}

A monoid is said to be {\em right rigid} if 
any two principal right ideals that intersect are comparable.
The notion of `rigidity' is defined in \cite{Cohn}.
Birget \cite{Birget} uses the terminology {\em $\mathscr{R}$-unambiguous}.
The proof of the following is easy or see Lemma~2.1 of \cite{Lawson2008a}.

\begin{lemma} 
A left cancellative monoid is right rigid if and only if it is equidivisible.
\end{lemma}

The following is proved as Theorem~2.6 of \cite{Lawson2008a}.

\begin{proposition} A monoid $S$ is a left Rees monoid if and only if it satisfies the following conditions:
\begin{description}
\item[{\rm (LR1)}] $S$ is a left cancellative monoid.

\item[{\rm (LR2)}]  $S$ is right rigid.

\item[{\rm (LR3)}] Each principal right ideal is properly contained in only a finite number of principal right ideals.

\item[{\rm (LR4)}] There is at least one non-invertible element.

\end{description}
\end{proposition}

\begin{remark}
{\em  Condition (LR4) is assumed in \cite{Lawson2008a} though not stated in this form.}
\end{remark}

\subsection{The structure of irreducible bimodules}

We begin with a result about maximal proper principal ideals in Levi monoids.

\begin{lemma}\label{le: molly} Let $S$ be a Levi monoid and let $a \in S$.
Then $SaS$ is a maximal proper principal ideal if and only if $a$ is an atom.
\end{lemma}
\begin{proof} Suppose that $a$ is not an atom.
If $a$ is invertible then $SaS = S$ and so $SaS$ is not a proper principal ideal.
If $a$ is not invertible and not an atom, then we may write $a = bc$ where $b$ is an atom and $c$ is not invertible.
It follows that $SaS \subseteq SbS$.
If $SaS = SbS$ then $a$ and $b$ would have the same length and so since $b$ is an atom it would follow that $a$ is an atom.
Thus  $SaS \subset SbS$ and $SaS$ is not maximal.
Conversely, suppose now that $a$ is an atom.
Suppose that $SaS \subseteq SbS$.
Then $a = xby$.
But $a$ has length 1 and so $xby$ has length 1.
If $b$ has length 1 then both $x$ and $y$ are invertible and $SaS = SbS$.
If $x$ or $y$ has length 1 then $b$ is invertible and $SbS = S$.
It follows that $SaS$ is a maximal proper principal ideal.
\end{proof}

A $G$-bimodule $X$ is said to be  {\em irreducible} if for all $x,y \in X$ we have that $y = gxh$ for some $g,h \in G$.
We shall write $(X,x)$ to mean that the irreducible bimodule $X$ is taken with respect to the point $x \in X$.

\begin{lemma}\label{le: max_ideal} 
A Levi monoid has a maximum proper principal ideal if and only if its associated bimodule of atoms is irreducible.
\end{lemma}
\begin{proof} Suppose first that the bimodule of atoms is irreducible.
Let $a$ be some fixed atom and let $b$ be any atom.
By assumption, $a = gbh$ where $g,h \in G$.
Thus $SaS = SbS$.
Let $SxS$ be any principal ideal where $x$ is not an atom and not invertible.
Then $x = by$ for some atom $b$.
Then $SxS \subseteq SbS = SaS$.
It follows that $SaS$ is the maximum proper principal ideal.
Conversely, suppose now that the monoid has a maximum proper principal ideal.
Let this be $SaS$.
Maximum is certainly maximal and so $a$ has to be an atom by Lemma~\ref{le: molly}.
Given any other atom $b$ we have $SbS \subseteq SaS$ and so $SbS = SaS$.
The result follows now by part (4) of Lemma~2.1.
\end{proof}

A Levi monoid whose associated bimodule of atoms is irreducible is itself said to be {\em irreducible}.
We shall now show how to construct all irreducible bimodules by generalizing a construction from \cite{N1}.
First some terminology.
A relation $\rho \subseteq A \times B$ will be said to {\em project onto both components}
if for each $a \in A$ there exists $b \in B$ such that $(a,b) \in \rho$
and for each $b \in B$ there exists $a \in A$ such that $(a,b) \in \rho$.

Let $X$ be an irreducible $G$-bimodule.
Choose $x \in X$.
Define the following two subsets of $G$
$$G_{x}^{+} = \{g \in G \colon gx \in xG\} \mbox{ and } G_{x}^{-} = \{g \in G \colon xg \in Gx \}$$
define also the subset, and so relation, $\gamma \subseteq G_{x}^{+} \times G_{x}^{-}$ by
$$g \, \gamma \, h \Leftrightarrow gx = xh.$$

\begin{lemma} With the above definitions, $G_{x}^{+}$ and $G_{x}^{-}$ are both subgroups of $G$ and 
$\gamma$ is a subgroup of $G_{x}^{+} \times G_{x}^{-}$ which projects onto both components.
Thus $\gamma$ is a subdirect product.
\end{lemma}

Now let $G$ be an arbitrary group, $H,K \leq G$ arbitrary subgroups and $\gamma \subseteq H \times K$
a subgroup that projects onto both components.
In lieu of better terminology, we shall call $(H,\gamma,K)$  {\em group data (derived from $G$)}.
Thus each point of an irreducible $G$-bimodule gives rise to group data.
The proof of the following is straightfoward.

\begin{lemma} Let $X$ be an irreducible $G$-bimodule and let $x,y \in X$. 
Let $\gamma$ be the group data arising from $x$ and $\delta$ the group data arising from $y$.
Suppose that  $x = uyv$.
Then if $g \in G_{x}^{+}$ and $h \in G_{x}^{-}$ are such that $(g,h) \in \gamma$ 
then $u^{-1}gu \in G_{y}^{+}$ and $vhv^{-1} \in G_{y}^{-}$ and $(u^{-1}gu,vhv^{-1}) \in \delta$.
\end{lemma}

Let  $(H_{1},\gamma_{1},K_{1})$  and  $(H_{2},\gamma_{2},K_{2})$ be group data where $H_{1},H_{2},K_{1},K_{2} \leq G$.
We say that they are {\em conjugate} if there are inner automorphisms $\alpha,\beta \colon G \rightarrow G$
such that $\alpha (H_{1}) = H_{2}$ and $\beta (K_{1}) = K_{2}$ and
$(h,k) \in \gamma_{1} \Leftrightarrow (\alpha (h),\beta (k)) \in \gamma_{2}$.
The above lemma may now be phrased as follows.

\begin{corollary}\label{cor: pinky} 
The group data arising from the choice of two points in an irreducible bimodule are conjugate.
\end{corollary}

\begin{lemma} Let $X$ and $Y$ be isomorphic irreducible bimodules over $G$ and choose $x \in X$ and $y \in Y$.
Then the group data associated with $(X,x)$ is conjugate to the group data associated with $(Y,y)$.
\end{lemma}
\begin{proof} It is immediate from the definition of an isomorphism between bimodules over $G$ that
$(X,x)$ and $(Y,\theta (x))$ have the same group data.
We now use Corollary~\ref{cor: pinky}, to deduce that the group data associated with $(Y,y)$ is conjugate to
the group data associated with $(Y,\theta (x))$.
\end{proof}

We shall show how to construct an irreducible $G$-bimodule from the group data $(H,\gamma,K)$.
On the set $G \times G$, define the relation $\equiv$ as follows
$$(g_{1},h_{1}) \equiv (g_{2},h_{2}) \Leftrightarrow g_{2}^{-1}g_{1} \in H, h_{2}h_{1}^{-1} \in K, 
\mbox{and }
g_{2}^{-1}g_{1} \, \gamma \, h_{2}h_{1}^{-1}.$$
The proof that $\equiv$ is an equivalence relation follows readily from the assumption that $H$, $K$ and $\gamma$
are all subgroups.
We denote the $\equiv$-equivalence class containing the pair $(g,h)$ by $[g,h]$.
The set of $\equiv$-equivalence classes is denoted by $\overline{X} = \overline{X}(H,\gamma,K)$.
Now the set $G \times G$ is a $G$-bimodule in the obvious way, and 
the equivalence $\equiv$ satisfies the following two conditions for any $g \in G$
$$(g_{1},h_{1}) \equiv (g_{2},h_{2}) \mbox{ implies } (gg_{1},h_{1}) \equiv (gg_{2},h_{2})$$
and 
$$(g_{1},h_{1}) \equiv (g_{2},h_{2}) \mbox{ implies } (g_{1},h_{1}g) \equiv (g_{2},h_{2}g).$$
We may therefore turn $\overline{X}$ into a $G$-bimodule by defining
$$g[g_{1},h_{1}] = [gg_{1},h_{1}] \mbox{ and } [g_{1},h_{1}]g = [g_{1},h_{1}g].$$

Let $x = [1,1] \in \overline{X}$.
Observe that $[g,h] = g[1,1]h$ and so the bimodule $\overline{X}$ is irreducible.
We shall calculate the subgroup of $G$ that consists of those elements $g \in G$ such that $gx \in xG$.
Let $h \in H$.
By assumption, there is $k \in K$ such that $h \, \gamma \, k$.
Now 
$$h[1,1] = [h,1] = [1,k] = [1,1]k$$
and so $H \subseteq G_{x}^{+}$.
It is easy to check that this is equality.
Together with the dual result we have proved part (1) of the following.

\begin{lemma} \mbox{}
\begin{enumerate}
\item From group data $(H,\gamma,K)$ we may construct an irreducible $G$-bimodule $\overline{X}(H,\gamma,K)$.
Furthermore, there is a point of $\overline{X}(H,\gamma,K)$ whose associated group data is  $(H,\gamma,K)$.

\item Let  $(H_{1},\gamma_{1},K_{1})$  and  $(H_{2},\gamma_{2},K_{2})$ be conjugate group data where $H_{1},H_{2},K_{1},K_{2} \leq G$
and the inner automorphisms  $\alpha,\beta \colon G \rightarrow G$
such that $\alpha (H_{1}) = H_{2}$ and $\beta (K_{1}) = K_{2}$ and
$(h,k) \in \gamma_{1} \Leftrightarrow (\alpha (h),\beta (k)) \in \gamma_{2}$
are given by $\alpha (g) = a^{-1}ga$ and $\beta (g) = b^{-1}gb$.
Then the bimodules $\overline{X}(H_{1},\gamma_{1},K_{1})$  and $\overline{X}(H_{2},\gamma_{2},K_{2})$ are isomorphic.

\end{enumerate}
\end{lemma}
\begin{proof} It only remains to prove (2).
Define $\theta \colon \overline{X}(H_{1},\gamma_{1},K_{1}) \rightarrow \overline{X}(H_{2},\gamma_{2},K_{2})$ 
by $\theta [g,h] = [ga,b^{-1}h]$.
Then $\theta$ defines an isomorphism of $G$-bimodules.
\end{proof}

The key result is the following.

\begin{proposition}\label{prop: perky} 
Let $X$ be an irreducible $G$-bimodule.
Choose $x \in X$ and let the group data associated with $x$ be $(H,\gamma,K)$.
Then $X$ is isomorphic to $\overline{X}(H,\gamma,K)$ as a $G$-bimodule.
\end{proposition}
\begin{proof} We define a function $\theta \colon X \rightarrow \overline{X}$ by
$\theta (y) = [g_{1},h_{1}]$ if $y = g_{1}xh_{1}$.
This is well-defined because 
$g_{1}xh_{1} = g_{2}xh_{2}$ if and only if $g_{2}^{-1}g_{1}x = xh_{2}h_{1}^{-1}$
if and only if $g_{2}^{-1}g_{1} \in H$, $h_{2}h_{1}^{-1} \in K$ and $g_{2}g_{1}^{-1} \, \gamma \, h_{2}h_{1}^{-1}$.
It follows that $(g_{1},h_{1}) \equiv (g_{2},h_{2})$.
It is now routine to check that $\theta$ is actually an isomorphism of $G$-bimodules.
\end{proof}

The nature of group data is a little elusive but we shall now characterize them in a way that is more intuitive.
A {\em divisor} of a group $G$ is a homomorphic image of a subgroup of $G$.
Thus it is a group that is isomorphic to one of the form $H/N$ where $H$ is a subgroup of $G$ and $N$ is a normal subgroup of $H$.
The following lemma can easily be derived from Goursat's lemma \cite{G}, 
though we give the details\footnote{ The authors are grateful to Zoe O'Connor (Heriot-Watt) for providing a reference for what we had originally taken to be folklore.}.

\begin{lemma} Let $G$ be a group.
Then group data derived from $G$ corresponds to isomorphisms between divisors of $G$.
\end{lemma}
\begin{proof}
Let $(H,\gamma,K)$ be group data where $H,K \leq G$.
Define
$$N_{H} = \{ h \in H \colon (h,1) \in \gamma \}
\mbox{ and } 
N_{K} = \{ k \in K \colon (1,k) \in \gamma \}.$$
Then $N_{H} \trianglelefteq H$ and $N_{K} \trianglelefteq K$.
We denote by $\alpha \colon H \rightarrow H/N_{H}$ 
and
$\beta \colon K \rightarrow K/N_{K}$
the associated natural maps.
Define $\theta \colon H/N_{H} \rightarrow K/N_{K}$
by $hN_{H} \mapsto kN_{K}$ if $(h,k) \in \gamma$.
Then in fact $\theta$ is an isomorphism.
Observe that $(h,k) \in \gamma$ if and only if $\theta (\alpha (h)) = \beta (k)$.

Conversely, 
let 
$N_{H} \trianglelefteq H$ and $N_{K} \trianglelefteq K$
and denote by 
$\alpha \colon H \rightarrow H/N_{H}$ 
and
$\beta \colon K \rightarrow K/N_{K}$
the associated natural maps
and let 
$\theta \colon H/N_{H} \rightarrow K/N_{K}$
be an isomorphism.
Define $\gamma \subseteq H \times K$ by
$(h,k) \in \gamma$ if and only if $\theta (\alpha (h)) = \beta (k)$.
It is routine to check that $\gamma$ is a subgroup of $H \times K$ and projects onto each component.
\end{proof}

We shall examine some special cases of our construction but first, we need some definitions.
A {\em partial endomorphism} $\phi$ of a group $G$ is a surjective homomorphism $\phi \colon H \rightarrow K$
where $H$ and $K$ are subgroups of $G$.
If $H$ has finite index in $G$ it is usual to call it a {\em virtual endomorphism}.
The subgroup $H$ is called the {\em domain of definition} of $\phi$.
If $\phi$ is also an isomorphism, we say that it is a {\em partial automorphism} of $G$.

\begin{lemma} Let $X$ be an irreducible $G$-bimodule.
Choose $x \in X$ and let $(H,\gamma,K)$ be the associated group data.
\begin{enumerate}

\item The right $G$-action is free if and only if $\gamma$ is the graph of a partial endomorphism of $G$ from $H$ onto $K$.

\item  The left and right $G$-actions are free if and only if $\gamma$ is the graph of a partial automorphism of $G$ from $H$ onto $K$.

\end{enumerate}
\end{lemma} 
\begin{proof} (1) Suppose first that the right $G$-action is free.
Let $(g,h_{1}),(g,h_{2}) \in \gamma$.
Then
$gx = xh_{1} = xh_{2}$.
Thus $xh_{1}h_{2}^{-1} = x$ and so by assumption $h_{1}h_{2}^{-1} = 1$ giving $h_{1} = h_{2}$.
It follows that $\gamma$ is the graph of a function.
To prove the converse, assume that $\gamma$ is the graph of a function.
Suppose that $xh = x$.
Then $(1,h) \in \gamma$.
But $(1,1) \in \gamma$.
It follows that $h = 1$ and the right action is free.
The proof of (2) is now immediate by symmetry.
\end{proof}

The following theorem pulls together what we have proved in this section and the previous one.

\begin{theorem} \mbox{}
\begin{enumerate}

\item Each irreducible Levi monoid with group of units $G$ can be constructed from an isomorphism between two divisors of $G$. 

\item Each irreducible left Rees monoid with group of units $G$ can be constructed from a partial endomorphism of $G$.

\item Each irreducible Rees monoid with group of units $G$ can be constructed from a partial automorphism of $G$.

\end{enumerate}
\end{theorem}

\section{Left Rees monoids and self-similar group actions}

For the remainder of this paper, we restrict our attention to left Rees monoids.
In this section, we shall recall the main results from  \cite{Lawson2008a} that will be needed and then we shall examine some simple
consequences.
Our main tool is the {\em Zappa-Sz\'ep} product of two monoids (see \cite{Brin} and the references there) which we shall define below.

\subsection{The correspondence}

Let $G$ be a group, $X$ a set, $G \times X^{\ast} \rightarrow X^{\ast}$ an operation, called the {\em action}, denoted by
$(g,x) \mapsto g \cdot x$, and $G \times X^{\ast} \rightarrow G$ an operation, called the {\em restriction}, denoted by $(g,x) \mapsto g|_{x}$,
such that the following eight axioms hold:
\begin{description}
\item[{\rm (SS1)}] $1 \cdot x = x$.
\item[{\rm (SS2)}] $(gh) \cdot x = g \cdot (h \cdot x)$.
\item[{\rm (SS3)}] $g \cdot 1 = 1$.
\item[{\rm (SS4)}] $g \cdot (xy) = (g \cdot x)(g|_{x} \cdot y)$.
\item[{\rm (SS5)}] $g|_{1} = g$.
\item[{\rm (SS6)}] $g|_{xy} = (g|_{x})|_{y}$.
\item[{\rm (SS7)}] $1|_{x} = 1$.
\item[{\rm (SS8)}] $(gh)|_{x} = g|_{h \cdot x} h|_{x}$.
\end{description}
Then we say that there is a {\em self-similar action of the group $G$ on the free monoid $X^{\ast}$}.
When we refer to a `self-similar group action $(G,X)$', we shall assume that the action and restriction have been chosen and are fixed.
It is easy to show that such an action is {\em length-preserving}, in the sense that $|g \cdot x| = |x|$ for all $x \in X^{\ast}$,
and {\em prefix-preserving}, in the sense that $x \preceq y$ implies that $g \cdot x \preceq g \cdot y$. 
If $x \in G$ then $G_{x}$ is the stabilizer of $x$ in $G$ with respect to the action and so a subgroup of $G$.

\begin{remark} {\em Our definition of what is meant by a self-similar group action is more general
than the one most visible in \cite{N1}.
We shall explain in what way below.}
\end{remark}

\begin{remark} {\em The axioms above fall into three groups.
Axioms (SS1), (SS2) and (SS3) say that the group $G$ acts on the left on the pointed set $(X^{\ast},1)$
fixing the distinguished element $1 = \varepsilon$.
Axioms (SS5), (SS6) and (SS7) say that the monoid $X^{\ast}$ acts on the right on the pointed set $(G,1)$
fixing the distinguished element $1$.
Axioms (SS4) and (SS8) link the two actions together.
The former shows how the left group action interacts with the monoid product in $X^{\ast}$,
whereas the latter shows how the right monoid action interacts with the group product in $G$.
If we define $g \circ x = g|_{x}$ then there is a perfect left-right symmetry between $\cdot$ and $\circ$ in the axioms.
However, there is an asymmetry in that $G$ is a group whereas $X^{\ast}$ is a cancellative monoid.}
\end{remark}

The following theorem was proved in detail in \cite{Lawson2008a} and so we shall only sketch it here.

\begin{theorem}\label{the: correspondence} 
There is a correspondence between the class of self-similar group actions and left Rees monoids.
\end{theorem}
\begin{proof}(Sketch)
Let $S$ be a left Rees monoid, with group of units $G$,
let $X$ be a transversal of the generators of the maximal proper principal right ideals,
and denote by $X^{\ast}$ the submonoid generated by the set $X$.
By part (5) of Lemma~\ref{le: bradbury}, we have that $S = X^{\ast}G$.
Then in fact $X^{\ast}$ is free,  and each element of $S$ can be written uniquely as a product of an element of $X^{\ast}$ and an element of $G$.
Let $g \in G$ and $x \in X^{\ast}$.
Then $gx \in S$ and so can be written uniquely in the form
$gx = x'g'$ where $x' \in X^{\ast}$ and $g' \in G$.
Define $x' = g \cdot x$ and $g' = g|_{x}$.
Then it is easy to check that this defines a self-similar action of $G$ on $X^{\ast}$.

Let $(G,X)$ be an arbitrary self-similar group action.
On the set $X^{\ast} \times G$ define a binary operation by
$$(x,g)(y,h) = (x(g \cdot y), g|_{y}h).$$ 
Then $X^{\ast} \times G$ is a left Rees monoid containing copies of $X^{\ast}$ and $G$
such that $X^{\ast} \times G$ can be written as a unique product of these copies.
This monoid is called the {\em Zappa-Sz\'ep product of $X^{\ast}$ and $G$} and is denoted $X^{\ast} \bowtie G$.

It follows that a monoid is a left Rees monoid 
if and only if it is isomorphic to a Zappa-Sz\'ep
product of a free monoid by a group. 
In turn, Zappa-Sz\'ep products of free monoids by groups determine,
and are determined by, self-similar group actions.
We have therefore set up a correspondence between left Rees monoids
and self-similar group actions in which
each determines the other up to isomorphism.
\end{proof}

\begin{remark}
{\em If we replace $G$ and $X^{\ast}$ by arbitrary monoids $S$ and $T$ respectively, then the axioms (SS1)--(SS8)
define the Zappa-Sz\'ep product of $S$ and $T$ denoted by $T \bowtie S$.
This monoid contains isomorphic copies of $S$ and $T$ and, by abusing notation a little,
we may write $T \bowtie S$ as simply $TS$.
Conversely, if a monoid $U$ contains submonoids $S$ and $T$ such that $U = TS$, uniquely,
then left and right actions are determined by deconstructing the associativity law and the properties of the identity.
It can then be proved that $U$ is isomorphic to  $T \bowtie S$.
If $S$ and $T$ are both groups then their Zappa-Sz\'ep product is also a group.}
\end{remark}

In a left Rees monoid $S$, the set $X$ is a transversal of the generators of the maximal proper principal right ideals.
Any such set will be called a {\em (left) basis} of the left Rees monoid $S$.
As long as the group of units of $S$ is not trivial, there will be many such bases.

\begin{remark} {\em If $S$ is a left Rees monoid with group of units $G$ then we write $S = X^{\ast}G$ 
to mean that a basis $X$ has been chosen and the data determining a self-similar action of $G$ on $X^{\ast}$.}
\end{remark}

If $S = X^{\ast}G$ is a left Rees monoid,
we may define the normalized length function $\lambda \colon S \rightarrow \mathbb{N}$ by
$\lambda (xg) = \left| x \right|$.

Let $S$ be a left Rees monoid with group of units $G$.
Define
$$K(S) = \{g \in G \colon \: gs \in sG \mbox{ for all } s \in S \},$$
a definition due to Rees \cite{Rees}.
This is a normal subgroup of $G(S)$
which we call the {\em kernel} of the left Rees monoid.
It can be checked that $K(S) = \bigcap_{x \in X^{\ast}}G_{x}$.

\begin{remark} 
{\em We have used the term `kernel' because this definition of Rees
given above agrees with the definition given in \cite{N1} on page 43.}
\end{remark}

\begin{remark} 
{\em Let $S$ be a left Rees monoid.
A subgroup $N \leq G$ is said to be a {\em right normal divisor}
if for all $s \in S$ we have that $Ns \subseteq sN$.
Such subgroups are clearly normal and subgroups of the kernel.
It can be shown, see Lemma~4.4 of \cite{Lawson2008a}, that $g \in N$ if and only if
for all $x \in X^{\ast}$ we have that $g \cdot x = x$ and $g|_{x} \in N$.
The kernel is therefore the largest right normal divisor.
These definitions go back to Rees's paper \cite{Rees}. 
We shall return to the significance of right normal divisors in the next section.}
\end{remark}

Left Rees monoids $S$ for which $K(S) = \{1 \}$ are said to be {\em fundamental}.\footnote{The inverse monoid associated with such a left Rees monoid is 
fundamental in the usual sense of inverse semigroup theory.}
We therefore have the following.

\begin{corollary}
A left Rees monoid is fundamental iff the corresponding group action is faithful.
\end{corollary}

\begin{remark} {\em We now pick up our remark above on the relationship between our definition of self-similar
group actions and the one to be found in \cite{N1}.
In the theory of self-similar group actions,
the actions are usually assumed to be faithful;
see, for example, Definition~1.51 of \cite{N1}.
However, in places in \cite{N1}, it is not assumed to be faithful: such as in the remark following Definition~2.11.1.
This is only a mild inconsistency, but it is, in any event, rectified by our general definition of a self-similar group action.
In the case where the action of $G$ on $X^{\ast}$ is faithful, the axioms (SS1)--(SS8) are not independent:
axioms (SS5)--(SS8) can be derived from (SS1)--(SS4).
The axioms (SS1)--(SS4) then constitute the usual, classical, definition of a faithful self-similar group action.}
\end{remark}

\subsection{The arithmetic of left Rees monoids}\setcounter{theorem}{0}

The goal of this section is to explore the way in which  the structure of a left Rees monoid reflects the properties of its associated self-similar group action.
We begin with a straightforward portmanteau lemma.

\begin{lemma}\label{le: two} Let $(G,X)$ be a self-similar group action.
\begin{enumerate}

\item $(g|_{x})^{-1} = g^{-1}|_{g \cdot x}$ for all $x \in X^{\ast}$ and $g \in G$.

\item The function $\phi_{x} \colon \: G_{x} \rightarrow G$ given by $g \mapsto g|_{x}$ is a homomorphism.

\item Let $y = g \cdot x$. 
Then $G_{y} = gG_{x}g^{-1}$
and
$$\phi_{y}(h) = g|_{x}\phi_{x}(g^{-1}hg) (g|_{x})^{-1}.$$

\item If $\phi_{x}$ is injective then $\phi_{g \cdot x}$ is injective.

\item $\phi_{x}$ is injective for all $x \in X$ iff $\phi_{x}$ is injective for all $x \in X^{\ast}$.

\item The function from $G$ to $G$ defined by $g \mapsto g|_{x}$ is injective for all $x \in X$ iff it is injective for all $x \in X^{\ast}$.

\item The function from $G$ to $G$ defined by $g \mapsto g|_{x}$ is injective for all $x \in X$ iff for all $x \in X$, 
if $g|_{x} = 1$ then $g = 1$.

\item $\phi_{x}$ surjective for all $x \in X$ iff $\phi_{x}$ is surjective for all $x \in X^{\ast}$.

\item The function from $G$ to $G$ given by $g \mapsto g|_{x}$ is surjective for all $x \in X$ iff it is surjective for all $x \in X^{\ast}$.

\end{enumerate}
\end{lemma}
\begin{proof}
(1) This was proved in \cite{Lawson2008a}.

(2) Let $g,h \in G_{x}$.
Then 
$$\phi_{x}(gh) = (gh)|_{x} = g|_{h \cdot x} h|_{x} = g|_{x} h|_{x} = \phi_{x}(g) \phi_{x}(h),$$
using (SS8), as required.

(3) We have that $h \cdot y = y$ iff $h \cdot (g \cdot x) = g \cdot x$ iff $g^{-1}hg \cdot x = x$ iff $g^{-1}hg \in G_{x}$.
Hence iff $h \in gG_{x}g^{-1}$.
The proof of the second claim follows by calculating $\phi_{x}(g^{-1}hg)$ using (SS8) and (1) above.

(4) This follows by (3) above.

(5) We obviously need only prove one direction.
We prove the result by induction on the length of $y$.
The result is true for strings of length one by assumption.
We assume the result true for strings of length $n$.
We now prove it for strings of length $n+1$.
Let $y$ be of length $n+1$.
Then $y = zx$ where $z$ has length $n$ and $x$ has length one.
We prove that $\phi_{y}$ is injective on $G_{y}$.
Let $h,k \in G_{y}$.
Then $h \cdot y = y = k \cdot y$.
By comparing lengths, it follows that 
$h \cdot z = z = k \cdot z$
and
$h|_{z} \cdot x = x = k|_{z} \cdot x$.
Suppose that $\phi_{y}(h) = \phi_{y}(k)$.
Then $h|_{y} = k|_{y}$.
By axiom (SS6), we have that
$(h|_{z})|_{x} = (k|_{z})|_{x}$.
But $h|_{z}, k|_{z} \in G_{x}$, and so
by injectivity for letters $h|_{z} = k|_{z}$.
Also $h,k \in G_{z}$, and so by the induction hypothesis $h = k$, as required.

(6) Again, only one direction needs to be proved and follows by induction using axiom (SS6).

(7) One direction is clear.
We prove the other direction.
Suppose that for all $x \in X$, if $g|_{x} = 1$ then $g = 1$.
We prove that the function from $G$ to $G$ defined by $g \mapsto g|_{x}$ is injective for all $x \in X$.
Suppose that $g|_{x} = h|_{x}$.
Then $g|_{x}(h|_{x})^{-1} = 1$.
By (1) above, $(h|_{x})^{-1} = h^{-1}|_{h \cdot x}$.
Put $y = h \cdot x$.
Then 
$$1 = g|_{x}(h|_{x})^{-1} = (g|_{h^{-1} \cdot y})( h^{-1}|_{y}) = (gh^{-1})|_{y}$$
by (SS8).
By assumption $gh^{-1} = 1$ and so $g = h$.

(8) Only one direction needs to be proved.
Let $y$ be a string of length $n+1$.
Then $y = zx$ where $x$ is a letter and $z$ has length $n$.
Let $g \in G$.
Then because $\phi_{x}$ is surjective, 
there exists $h \in G_{x}$ such that $\phi_{x}(h) = g$.
By the induction hypothesis, there exists $k \in G_{z}$ such that $\phi_{z}(k) = h$.
We now calculate 
$$k \cdot y = k \cdot (zx) = (k \cdot z)(k|_{z} \cdot x) = zx = y.$$
Thus $k \in G_{y}$ and $\phi_{y}(k) = k|_{zx} = (k|_{z})|_{x} = h|_{x} = g$,
as required using axiom (SS6).

(9) Only one direction needs to be proved and follows by induction and axiom (SS6).
\end{proof}

The structure of the principal right, left and two-sided ideals will play an important role in our calculations.

\begin{lemma}\label{le: greens} Let $S = X^{\ast}G$ be a left Rees monoid. 
\begin{enumerate}

\item $xh \,\mathscr{R}\, yk$ iff $x = y$. 

\item $xg\, \mathscr{L} \,yh$ iff there exists an invertible element $k$ such that
$k \cdot y = x$ and $k|_{y} = gh^{-1}$.

\item $xg \, \mathscr{H} \, yh$
iff
$x = y$
and there exists $k \in G_{x}$ such that $k|_{x} = gh^{-1}$.

\item  Let $x,y \in X^{\ast}$. 
Then $x \, \mathscr{J} \, y$ 
if and only if 
$x = g \cdot y$ for some $g \in G$.

\end{enumerate}
\end{lemma}
\begin{proof} (1) Straightfoward.

(2) Suppose that such a $k$ exists.
Then
$$k(yh) = (k \cdot y)k|_{y}h = xgh^{-1}h = xg.$$
The result now follows by Lemma~\ref{le: Greens_relations}.
Conversely, suppose that $xg \, \mathscr{L} \, yh$.
Then by Lemma~\ref{le: Greens_relations}, there exists an invertible element $k$ such that $k(yh) = xg$.
The result is now immediate.

(3) Immediate by (1) and (2).

(4)  Suppose that $x \, \mathscr{J} \, y$. 
Then there are group elements $g,h \in G$ such that $x = gyh$.
But $gyh = (g \cdot y)g|_{y}h$.
By uniqueness, we have that $x = g \cdot y$.
Conversely, suppose that $x = g \cdot y$.
Observe that $SgyS = S(g \cdot y)g|_{y}S = S(g \cdot y)S = SxS$.
Thus $SxS = SyS$ and so $x\, \mathscr{J} \, y$, as required.
\end{proof}

We shall now relate the behaviour of the action of $G$ on $X^{\ast}$ with the properties of the principal two-sided ideals.
We know that this action is length-preserving.
If the action of $G$ on $X^{n}$ is transitive for all $n$, then we say that the action is {\em level transitive}.

\begin{proposition}\label{prop: level_transitivity} Let $S = X^{\ast}G$ be a left Rees monoid.
\begin{enumerate}

\item There is a bijective correspondence between the orbits arising from the action of $G$ on $X$ and the maximal proper principal ideals.

\item The action of $G$ on $X$ is transitive if and only if $S$ has a maximum proper principal two-sided ideal if and only if $S$ is irreducible.

\item The action of $G$ on $X^{\ast}$ is level-transitive if and only if $S/\mathscr{J}$ is an infinite descending chain
order-isomorphic to the natural numbers with their reverse order.
This is also equivalent to $x\, \mathscr{J} \, y$ iff $|x| = |y|$. 

\end{enumerate}
\end{proposition}
\begin{proof} (1) Let $s \in S$.
In Lemma~\ref{le: max_ideal},  
we saw that $SsS$ is a maximal proper principal ideal if and only if $s$ is an atom.
The maximal proper principal ideals are therefore of the form $SxS$ where $x \in X$.
We have that $SxS = SyS$ if and only if $x = g \cdot y$ for some $g \in G$.
The result now follows.

(2).  The first statement follows by (1) above.
The second statement follows by Lemma~\ref{le: max_ideal}.

(3) Suppose that the action is level-transitive.
Then $x\, \mathscr{J}\, y$ if and only if $|x| = |y|$. 
We prove that the principal two-sided ideals form an infinite desecending chain.
Let $I_{n} = SxS$ where $x \in X^{n}$ is a string of length $n$.
By our assumption, we have that $I_{n} = SyS$ where $y$ is any string of length $n$.
If now $|x| = n+1$ and $x = yz$ where $|y| = n$,
then $SxS \subseteq SyS$.
Hence $I_{n+1} \subseteq I_{n}$, and we have our descending chain.

Assume now that the principal two-sided ideals form an infinite descending chain
order-isomorphic to the natural numbers with their reverse order.
We shall prove that the action is level-transitive.
We denote the principal two-sided ideals by $I_{n} = Sx_{n}S$ for some $x_{n} \in X^{\ast}$
where $I_{n+1} \subseteq I_{n}$.
Since $x_{n+1} \in Sx_{n}S$ we know that $|x_{n+1}| \geq |x_{n}|$ by Lemma~\ref{le: Greens_relations}.
We also know that $S = I_{0}$.
By Lemma~\ref{le: molly}, we have that $x_{1}$, the generator of $I_{1}$ has length one.
Suppose now that for all $m \leq n$ we have proved that $Sx_{m}S$ is equal to $SxS$ where $x$ is any string of length $m$.
We now prove the same for $I_{n+1}$.
We prove first that $x_{n+1}$ has length $n+1$.
It cannot have length $n$ or less thus we can write it
as $x_{n+1} = xyz$ where $x$ has length $n$, and $y$ has length one.
Then $Sx_{n+1}S \subseteq SxyS$.
If they are not equal, then $SxyS$ has got to equal one of the earlier ideals in the chain.
But that would mean $xy$ would have length at most $n$ which is a contradiction.
Thus $Sx_{n+1}S = SxyS$.
It follows that $x_{n+1}$ has length $n+1$.
Now let $SxS$ be any ideal where $x$ has length $n+1$.
It cannot be equal to any of the earlier ideals and so it is equal either to $I_{n+1}$ or to a later ideal in the chain.
In any event, $SxS \subseteq Sx_{n+1}S$.
But $x$ and $x_{n+1}$ have the same length and so we have that $x = g \cdot x_{n+1}$.
We conclude that $I_{n}$ is generated by any element of $X^{n}$ and that $G$ acts transitively on $X^{n}$ for any $n$.
\end{proof}

We now turn to the $\mathscr{R}$-relation.
Observe that if $x,y \in X^{\ast}$ then $xS \subseteq yS$ iff $x = yz$ for some $z \in X^{\ast}$.
Combined with Lemma~\ref{le:  greens}, this tells us that the partially ordered set $S/\mathscr{R}$ of
$\mathscr{R}$-classes is order-isomorphic to the set $X^{\ast}$ equipped with the prefix ordering.
We shall now consider the relationship between the $\mathscr{R}$-relation on left Rees monoids and certain kinds of congruences.

\begin{proposition} Let $\theta \colon S \rightarrow T$ be a monoid homomorphism between left Rees monoids
where $S$ has group of units $G$ and $T$ has group of units $H$.
Let  $S = X^{\ast}G$ for some $X$.
Suppose in addition that the kernel of $\theta$ is contained in $\mathscr{R}$.
Define $N = \{g \in G \colon \theta (g) = 1\}$.
Then $N$ is a right normal divisor.
In addition, if $\theta (s_{1}) = \theta (s_{2})$ then there exists $n \in N$
such that $s_{1} = s_{2}n$.
If $\theta$ is surjective then $\theta$ is atom preserving.
\end{proposition} 
\begin{proof} Let $s \in S$ be an arbitrary element.
Then $s = xg$.
Let $h \in N$.
Then $\theta (hs) = \theta (s)$.
By assumption, $hs \, \mathscr{R} \, s$.
But $hs = hxg = (h \cdot x)h|_{x}g$ and so by Lemma~\ref{le: greens},
we have that $x = h \cdot x$.
From $\theta (hs) = \theta (s)$ we get that
$\theta (x h|_{x}g) = \theta (xg)$.
Thus by left cancellation, we get that 
$\theta (h|_{x}g) = \theta (g)$.
The element $\theta (g)$ is invertible and so $\theta (h|_{x}) = 1$.
It follows that $h|_{x} \in N$.
Thus $hs = x h|_{x}g = (xg)g^{-1}h|_{x}g = sg^{-1}h|_{x}g$.
But $N$ is certainly a normal subgroup of $G$ and so $N$ is a right normal divisor.

Let $\theta (s_{1}) = \theta (s_{2})$ where $s_{1} = x_{1}g_{1}$ and let $s_{2} = x_{2}g_{2}$.
By assumption, $s_{1} \, \mathscr{R} \, s_{2}$ and so $x_{1} = x_{2} = x$, say.
By left cancellation, we have that $\theta (g_{1}) = \theta (g_{2})$.
It follows that $g_{2} = g_{1}n$ for some $n \in N$.

Suppose that $\theta$ is surjective.
Let $a \in S$ be an atom and suppose that $\theta (a) = \theta (b)\theta (c)$.
Then $a \, \mathscr{R} \, bc$.
In particular, $a$ and $bc$ have the same length.
It follows that since $a$ is an atom so too is $bc$.
This means that $b$ is an atom and $c$ is invertible or vice-versa.
In the first case, $\theta (c)$ is invertible and in the second case $\theta (b)$ is invertible.
Thus $\theta (a)$ is an atom.
\end{proof}

\begin{remark} {\em It follows from the classical work by Rees \cite{Rees},
that right normal divisors determine and are determined by those monoid congruences on
left Rees monoids that are contained in the $\mathscr{R}$-relation.
In fact, if $S = X^{\ast}G$ and $N \subseteq G$ is a right normal divisor
then the quotient of $S$ determined by $N$ is isomorphic to $X^{\ast}G/N$.
It follows that the congruence determined by the kernel of $S$ is the
maximum monoid congruence contained in the $\mathscr{R}$-relation.
Thus Proposition~2.7.4 of \cite{N1} is a version of Rees's theory of right normal divisors.}
\end{remark}

The relation $\mathscr{R}$ is always a {\em left} congruence.

\begin{proposition} The relation $\mathscr{R}$ is a right congruence, and so a congruence,
if and only if the action of $G$ on $X^{\ast}$ is trivial.
\end{proposition}
\begin{proof}
Suppose that $\mathscr{R}$ is a right congruence.
Let $g \in G$.
Then $gS = S = 1S$ and so $g \, \mathscr{R} \,1$.
Let $x \in X^{\ast}$ be arbitrary.
By assumption, $gx \, \mathscr{R} \, x$.
Thus $gxS = xS$.
There is therefore a unit $h$ such that $gx = xh$.
But $gx = (g \cdot x)g|_{x}$.
By uniqueness, $x = g \cdot x$, and so the action is trivial.

Conversely, suppose that the action is trivial.
Let $a \, \mathscr{R} \, b$. 
We prove that  $ac \, \mathscr{R} \, bc$.
Let $a = bg$ and $c = yh$.
Then
$$ac = bgc = bgyh = b(g \cdot y)g|_{y}h = byg|_{y}h = byh(h^{-1}g|_{y}h) = bc(h^{-1}g|_{y}h)$$
and so $ac \, \mathscr{R} \, bc$, as required.
\end{proof}

\begin{remark}{\em
Suppose that the action of $G$ on $X^{\ast}$ is trivial.
Define $G \times X^{\ast} \rightarrow G$
by $(x,g) \mapsto g|_{x} = g \circ x$.
Then
$g \circ 1 = g|_{1} = g$ by (SS5)
and
$g \circ (xy) = g|_{xy} = (g|_{x})|_{y} = (g \circ x) \circ y$ by (SS6)
and
$(gh) \circ x = (gh)|_{x} = g|_{h \cdot x}h|_{x} = (g \circ x)(h \circ x)$
by (SS8) and the fact that the action is trivial.
Thus the monoid $X^{\ast}$ acts on the group $G$ on the right by endomorphisms.
In this case, the product assumes the following form
$$(xg)(yh) = x(g \cdot y)g|_{y}h = xyg|_{y}h = xy(g \circ y)h.$$  
It follows that $S$ can be described as a semidirect product of two monoids.
We regard semidirect products of monoids as `degenerate' Zappa-Sz\'ep products.
Observe also that in this case $S/\mathscr{R}$ is a free monoid isomorphic to $X^{\ast}$.}
\end{remark}

We now touch on some properties of the $\mathscr{H}$-classes.

\begin{lemma}\label{le: H-structure} Let $S = X^{\ast}G$. 
Then $\left| H_{xg} \right| = \left| \mbox{\rm im} (\phi_{x}) \right|$.
\end{lemma}
\begin{proof} By Lemma~\ref{le:  greens},
we have that $yh \, \mathscr{H} \, xg$ if and only if $y = x$ and there exists $k \in G_{x}$ such that
$k|_{x} = gh^{-1}$. 
It follows that it consists of all elements $xh$ where $h \in \mbox{\rm im}(\phi_{x})g$.
The result now follows.
\end{proof}

We now extend some notation introduced earlier.
Let $S$ be a monoid with group of units $G$.
For any $s \in S$ define the following two subgroups of $G$
$$G_{s}^{+} = \{g \in G \colon gs \in sG\} \mbox{ and } G_{s}^{-} = \{g \in G \colon sg \in Gs \}.$$

\begin{lemma} Let $S = X^{\ast}G$ be a left Rees monoid.
If $s = xh$ then $G_{s}^{+} = G_{x}$.
\end{lemma} 
\begin{proof} Let $g \in G_{s}^{+}$.
Then $gs = sg'$ for some $g'$.
Hence $g(xh) = xhg'$.
Thus $(g \cdot x) g|_{x}h = xhg'$.
It follows that $g \cdot x = x$ and so $g \in G_{x}$.
Conversely, suppose that $g \in G_{x}$.
Then $g(xh) = (g \cdot x)g|_{x}h = x g|_{x}h$.
Hence $gs = sh^{-1}g|_{x}h$ and so $g \in G_{x}^{+}$.
\end{proof}

Left Rees monoids contain exactly one idempotent, the identity, and so contain only one group $\mathscr{H}$-class,
the group of units.
There are, however, groups associated with any $\mathscr{H}$-class whose definition we now recall.
Let $S$ be a monoid.
Let $a \in S$.
We are interested in all the elements $s$ such that $H_{a}s \subseteq H_{a}$.
Each such element $s$ determines a permutation of $H_{a}$ and the  {\em Sch\"utzenberger group} associated with the 
$\mathscr{H}$-class $H_{a}$ is then the group generated by these permuations and Sch\"utzenberger groups of $\mathscr{H}$-classes
belonging to the same $\mathscr{D}$-class (see below) are isomorphic \cite{Lallement1}.

\begin{lemma}\label{le: Schutz} The Sch\"utzenberger group associated with the $\mathscr{H}$-class $H_{xg}$ is 
$$g^{-1}\phi_{x}(G_{x})g.$$
It is also equal to $G_{xg}^{-}$.
\end{lemma}
\begin{proof}
We first calculate the right stabilizer of $H_{xg}$.
Observe that if $H_{xg}a = H_{xg}$ then $a$ has to be invertible.
If $l \in G$ and $H_{xg}l \subseteq H_{xg}$ then it is easy to check that $l \in g^{-1}\phi_{x}(G_{x})g$,
and that if  $l \in g^{-1}\phi_{x}(G_{x})g$ then $H_{xg}l \subseteq H_{xg}$.
Equality rather than inclusion follows from the fact that $g^{-1}\phi_{x}(G_{x})g$ is a subgroup of $G$.
It is now immediate that $g^{-1}\phi_{x}(G_{x})g$ {\em is} the Sch\"utzenberger group associated with the $\mathscr{H}$-class $H_{xg}$.

We now prove the second claim.
Let $h \in G$ be such that $(xg)h \in Gxg$.
For such an $h$ there exists $k \in G$ such that
$xgh = kxg$.
Thus $x = k \cdot x$ and $h = g^{-1}(k|_{x})g$.
Hence $h \in g^{-1}\phi_{x}(G_{x})g$.
Conversely, suppose that $h \in g^{-1}\phi_{x}(G_{x})g$.
Then $h = g^{-1}(k|_{x})g$ for some  $k \in G_{x}$.
Thus $(xg)h = k(xg)$.
\end{proof}

We now consider the properties of the maps $\phi_{x}$.

\begin{proposition}\label{prop: cancellative} Let $S = X^{\ast}G$ be a left Rees monoid.
Then the following are equivalent.
\begin{enumerate}

\item The functions $\phi_{x} \colon \: G_{x} \rightarrow G$ are injective for all $x \in X^{\ast}$.

\item The monoid $S$ is right cancellative (and so cancellative).

\end{enumerate}
\end{proposition}
\begin{proof}
(1)$\Rightarrow$(2).
We prove that $S$ is right cancellative.
Let $ab = cb$
where $a = wg$, $b = yk$ and $c = zl$.
We therefore get $w(g \cdot y) = z(l \cdot y)$ and $g|_{y}k = l|_{y}k$.
Since the action is length-preserving and by uniqueness we have that
$g \cdot y = l \cdot y$, $w = z$ and $g|_{y} = l|_{y}$.
Our result will follow if we can show that $g = l$.
Observe that $g^{-1}l \in G_{y}$.
Thus $\phi_{y}(g^{-1}l)$ is defined and equals $(g^{-1}l)|_{y}$.
But $(g^{-1}l)|_{y} = 1$ by Lemma~\ref{le: two}.
Thus $g = l$, as required.

(2)$\Rightarrow$(1).
We prove that $\phi_{x}$ is injective.
Let $g,h \in G_{x}$ and suppose that $\phi_{x}(g) = \phi_{x}(h)$.
Thus $g|_{x} = h|_{x}$.
Now $gx = (g \cdot x)g|_{x} = xg|_{x}$,
and $hx = (h \cdot x)h|_{x} = xh|_{x}$.
Thus $gx = hx$.
By right cancellation, $g = h$, and so $\phi_{x}$ is injective.
\end{proof}

\begin{remark} {\em By part (5) of Lemma~\ref{le: two}, all the functions $\phi_{x}$ are injective precisely
when the functions $\phi_{x}$ are injective where the $x$ range over all letters.
If the left Rees monoid is actually irreducible, then by part (4) of Lemma~\ref{le: two}, 
it is a Rees monoid if any one of the maps $\phi_{x}$ is injective where $x$ is a letter.}
\end{remark}

\begin{proposition}\label{prop: self-replicating} Let $S = X^{\ast}G$ be a left Rees monoid.
Then the following are equivalent.
\begin{enumerate}

\item The functions $\phi_{x} \colon \: G_{x} \rightarrow G$ are surjective for all $x \in X^{\ast}$.

\item $\mathscr{R} = \mathscr{H}$.

\item $\mathscr{R} \subseteq \mathscr{L}$.

\item $\mathscr{L} = \mathscr{D}$.

\item For all $x \in X^{\ast}$ and $g \in G$ we have that $xg \, \mathscr{L} \, x$.

\end{enumerate}
\end{proposition}
\begin{proof}
(1)$\Rightarrow$(2).
Suppose that all the functions $\phi_{x} \colon \: G_{x} \rightarrow G$ are surjective.
We have that $\mathscr{H} \subseteq \mathscr{R}$ always, so we need only establish the reverse inclusion.
We have that $xg \, \mathscr{R} \, xh$, by Lemma~\ref{le: greens}.
By assumption, there is $k \in G_{x}$ such that $\phi_{x}(k) = gh^{-1}$.
But by Lemma~\ref{le: greens}, this shows that $xg \, \mathscr{H} \, xh$, as required.

(2)$\Rightarrow$(3). $\mathscr{R} = \mathscr{H} = \mathscr{L} \cap \mathscr{R} \subseteq \mathscr{L}$.
Thus $\mathscr{R} \subseteq \mathscr{L}$.

(3)$\Rightarrow$(4). $\mathscr{L} \subseteq \mathscr{D} = \mathscr{R} \circ \mathscr{L} \subseteq \mathscr{L} \circ \mathscr{L} \subseteq \mathscr{L}$.
Thus $\mathscr{D} = \mathscr{L}$.

(4)$\Rightarrow$(5). By Lemma~\ref{le: greens}, we have $x \, \mathscr{R} \, xg$.
Thus $x \, \mathscr{D} \, xg$ and so $x \, \mathscr{L} \, xg$ by assumption.

(5) $\Rightarrow$ (1). 
Let $g \in G$ and $x \in X^{\ast}$.
By assumption, $xg \, \mathscr{L} \, x$.
By Lemma~\ref{le: greens}, there exists $k$ such that $k \cdot x = x$ and $k|_{x} = g$.
Thus $\phi_{x}(k) = g$ and so $\phi_{x}$ is surjective.
\end{proof}

\begin{lemma}\label{le: more_level_transitivity} Let $S = X^{\ast}G$ be a left Rees monoid.
Suppose that the functions $\phi_{x} \colon \: G_{x} \rightarrow G$ are surjective for all $x \in X^{\ast}$.
Then the action of $G$ on $X$ is transitive if and only if the action of $G$ on $X^{\ast}$ is level transitive.
\end{lemma}
\begin{proof} Only one direction needs proving.
We use induction. 
Suppose that $G$ is transitive on $X^{n}$.
We prove that it is transitive on $X^{n+1}$.
Let $u$ and $v$ be strings of length $n+1$.
Then $u = wx$ and $v = zy$ where $x,y \in X$.
By the induction hypothesis there is a $g \in G$ such that
$g \cdot w = z$.
Thus $g \cdot u = (g \cdot w)(g|_{w} \cdot x) = z(g|_{w} \cdot x)$.
By the transitivity of $G$ on $X$, there is $k \in G$ such that $k \cdot (g|_{w} \cdot x) = y$,
and by assumption, there is $h \in G_{z}$ such that $h|_{z} = k$.
Thus 
$h \cdot (z(g|_{w} \cdot x))  
=
(h \cdot z)(h|_{z} \cdot g|_{w} \cdot x)
=
zy
=
v
$,
and so $(hg) \cdot u = v$.
Thus the action is level transitive.
\end{proof}

A monoid $S$ is said to be {\em right reversible}
if for all $a,b \in S$ the set $Sa \cap Sb \neq \emptyset$. 

\begin{proposition}\label{prop: recurrent} Let $S = X^{\ast}G$ be a left Rees monoid.
Then the following are equivalent.
\begin{enumerate}

\item The action of $G$ on $X$ is transitive, and the functions $\phi_{x} \colon \: G_{x} \rightarrow G$ are surjective for all $x \in X^{\ast}$.

\item The monoid $S$ has a maximum proper principal left ideal.

\item  $S/\mathscr{L}$ is an infinite descending chain.

\item $S$ is right reversible.

\end{enumerate}
\end{proposition} 
\begin{proof}
(1)$\Rightarrow$(2). Let $x \in X$.
We prove that $Sx$ is the maximum principal left ideal.
Let $y \in X$.
Then by transitivity there is $g \in G$ such that $g \cdot x = y$.
Thus $gx = (g \cdot x)g|_{x} = yg|_{x}$.
By Proposition~\ref{prop: self-replicating}, we have that $yg|_{x} \, \mathscr{L} \, y$.
Thus $Sgx = Sx = Syg|_{x} = Sy$.
It follows that $Sx = Sy$ for all $y \in X$.
By Proposition~\ref{prop: self-replicating}, we have that $Sx = Syg$ for all $y \in X$ and $g \in G$.
Now let $uk$ be an arbitrary element of $S$ where $|u| > 1$.
Then $u = u'y$ where $|y| = 1$.
It follows that $Suk \subseteq Syk = Sx$, by the above.
We have proved that $Sx$ is indeed the maximum proper principal left ideal.

(2)$\Rightarrow$(1).
Let $Sxg$ be the maximum proper principal left ideal.
We prove first that $|x| = 1$.
Suppose not.
Then we can write $x = uy$ where $|y| = 1$ and $|u| \geq 1$.
Thus $Sxg = Suyg \subseteq Syg$.
By assumption $Sxg = Syg$.
It follows that $yg = sxg$ for some $s \in S$.
Now $|y| = 1$ and $|x| > 1$ and so $s$ would have to be invertible at worst
but then by uniqueness we get a contradiction by comparing the lengths
of the elements of $X^{\ast}$ on each side.
We deduce as claimed that $|x| = 1$.

Since $Sxg$ is a maximum principal left ideal we have that $Sx \subseteq Sxg$.
By length considerations again, there is an invertible element $k$ such that $x = kxg$.
But then $Sx = Skxg = Sxg$.
Hence the maximum principal left ideal has the form $Sx$ where $x \in X$.
Let $y \in X$.
Then $Sy \subseteq Sx$ and so by length considerations there is $k \in G$ such that
$y = kx$ and so $y = (k \cdot x)k|_{x}$.
By uniqueness $k|_{x} = 1$ and so $y = k \cdot x$.
We have therefore proved that $G$ is transitive on $X$.

Let $y \in X$ and $h \in G$.
Then $Syh \subseteq Sx$.
Thus there is an invertible element $k$ such that $yh = kx$.
Thus $Syh = Skx = Sx = Sy$.
We have therefore shown that $yh \, \mathscr{L} \, y$.
By Lemma~\ref{le: greens}, there exists $k \in G_{y}$ such that $\phi_{y}(k) = h$.
It follows that the function $\phi_{y} \colon \: G_{y} \rightarrow G$ is surjective for all $y \in X$.
By part (8) of Lemma~\ref{le: two}, 
it follows that $\phi_{u}$ is surjective for all $u \in X^{\ast}$.

(2)$\Rightarrow$(3). By Lemma~\ref{le: more_level_transitivity}, the action is also level transitive.
By Proposition~\ref{prop: level_transitivity}, 
this implies that the principal two-sided ideals form an infinite descending chain.
By Proposition~\ref{prop: self-replicating}, 
we have that $\mathscr{L} = \mathscr{D}$.
By Proposition~\ref{prop: self-replicating},
it follows that $x \, \mathscr{L} \, y$ if and only if $\left| x \right| = \left| y \right|$. 
Thus for any $x,y \in X^{n}$, we have that $Sx = Sy$.
Let $x$ be a string of length $n$.
Write $x = yz$ where $y$ has length $1$ and $z$ has length $n-1$.
We have that $Sx \subseteq Sz$.
It follows that the principal left ideals form an infinite descending chain.

(3)$\Rightarrow$(2). Immediate.

(1)$\Rightarrow$(4). 
By (2) above, the principal left ideals form a chain.
It follows that given any two such ideals one must be contained in the other and so the
semigroup is right reversible.

(4)$\Rightarrow$(3). 
Suppose that the semigroup is right reversible.
By assumption, any two principal left ideals must intersect and so one must be contained in the other
since left Rees monoids are also left rigid.
It follows that the principal left ideals form an infinite descending chain.
\end{proof}

A self-similar group action is {\em recurrent}
if $G$ is transitive on $X$, and $\phi_{x}$ is onto for any $x \in X$ \cite{N1}.
We shall also say that the corresponding left Rees monoid is {\em recurrent}.
Such monoids were, in fact, the subject of Rees's classic paper \cite{Rees}.
Rees's theory can now be phrased in our terms.

\begin{theorem}[Rees]\label{the: Rees} Let $G$ be a group and let $\alpha \colon G \rightarrow G$ be an endomorphism of $G$.
Put $S = \mathbb{N} \times G$ with product defined by 
$$(m,g)(n,h) = (m+n,\alpha^{n}(g)h).$$
Then $S$ is a left Rees monoid in which the principal right ideals form an infinite descending chain
order isomorphic to the natural numbers with their reverse order.
Furthermore, every  left Rees monoid in which the principal right ideals form such an infinite descending chain
is isomorphic to a monoid of the form $S$.
\end{theorem}

We denote the monoid described above by $\mathbb{N} \, \ast_{\alpha} \, G$.
The structure of {\em recurrent Rees monoids}  as defined above can be analysed in more detail. 
This is because such a monoid is a right Rees monoid in which the principal left ideals form an infinite descending chain.
We may therefore apply the dual form of Theorem~\ref{the: Rees}.
Let $G$ be a group and let $\alpha \colon \: G \rightarrow G$ be a homomorphism.
On the set $G \times \mathbb{N}$ define a product by
$$(g,p)(h,q) = (g\alpha^{p}(h),p+q).$$
Then this defines a right cancellative monoid in which the set of principal left ideals forms an infinite descending chain.
We denote this monoid by $G \ast_{\alpha} \mathbb{N}$.
The monoid $G \ast_{\alpha} \mathbb{N}$ is cancellative precisely when $\alpha$ is injective;
this follows by Proposition~\ref{prop: cancellative}, though it is also an easy direct calculation.

\begin{theorem}\label{the: recurrent_Rees} Let $G$ be a group and let $\alpha \colon \: G \rightarrow G$ be an injective homomorphism with image $H$.
Then  
$G \ast_{\alpha} \mathbb{N}$
is a recurrent Rees monoid having $\left| G : H \right|$ proper maximal principal right ideals
and every recurrent Rees monoid is isomorphic to one constructed in this way.
\end{theorem}

\begin{remark}{\em 
There is a connection between recurrent Rees monoids and strong representations of the polycyclic monoids \cite{JL1}.
Let $G$ be a group and $H$ a subgroup with finite index $n$.
Then the set of cosets $G/H$ forms a partition of $G$ where all the blocks have the same cardinality as $H$.
We shall concentrate on the case where $H$ is isomorphic to $G$.
Let $\alpha \colon G \rightarrow H$ be an isomorphism and let $G = \bigcup_{i=1}^{n} g_{i}H$ be a coset decomposition of $G$.
Then we can define bijections $\alpha_{i} \colon G \rightarrow g_{i}H$ by $\alpha_{i}(g) = g_{i}\alpha(g)$.
This is a strong representation of $P_{n}$ in $I(G)$ which we call a {\em strong affine representation}.
From the isomorphism  $\alpha^{-1} \colon H \rightarrow G$ we may construct an irreducible
Rees monoid $S$ in which $S/\mathscr{L}$ is an infinite descending chain.
There is a bijection between the sets of bases of $S$ and the coset decompositions $G = \bigcup_{i=1}^{n} g_{i}H$.
The data of such a Rees monoid and a basis may be used to construct a strong affine representation
$P_{n} \rightarrow I(G)$ to the symmetric inverse monoid on $G$ and every such representation arises in this way.}
\end{remark}

We now turn to the full restriction maps;
these are the maps $G \rightarrow G$ given by $g \mapsto g|_{x}$ where $x \in X^{\ast}$.

\begin{lemma}\label{le: yan} 
Let $S$ be a left Rees monoid.
Then the following are equivalent.
\begin{enumerate}
\item The functions from $G$ to itself given by $g \mapsto g|_{x}$ are injective for all $x \in X^{\ast}$. 

\item The monoid $S$ is cancellative and each $\mathscr{L}$-class contains at most one element from $X^{\ast}$.
\end{enumerate}
\end{lemma}
\begin{proof}
(1)$\Rightarrow$(2).
Suppose that the given functions are all injective.
Clearly the $\phi_{x}$ are injective and so $S$ is cancellative by Proposition~\ref{prop: cancellative}.
Suppose now that $x \, \mathscr{L} \, y$ where $x,y \in X^{\ast}$.
Then by Lemma~\ref{le: greens}, there exists a $k \in G$ such that
$k \cdot y = x$ and $k|_{y} = 1$.
By injectivity, we have that $k = 1$ and so $x = y$, as required.

(2)$\Rightarrow$(1).
Suppose that $S$ is cancellative and each $\mathscr{L}$-class contains at most one element from $X^{\ast}$. 
Let $g|_{x} = 1$ where $x \in X$.
Then by Lemma~\ref{le: greens}, we have that $g \cdot x \, \mathscr{L} \, x$.
But $g \cdot x \in X$ and so $g \cdot x = x$, by assumption.
Thus $g \in G_{x}$.
By Proposition~\ref{prop: cancellative}, we have that $g = 1$.
Thus by part (7) of Lemma~\ref{le: two},
the functions $g \mapsto g|_{x}$ are injective for all letters $x$.
It follows by part (6) of Lemma~\ref{le: two} that the functions $g \mapsto g|_{x}$ are
injective for all strings $x$.
\end{proof}

We now turn to surjectivity of the full restriction maps.

\begin{lemma}\label{le: tan}
Let $S$ be a left Rees monoid.
Then the following are equivalent.
\begin{enumerate}
\item The functions from $G$ to itself given by $g \mapsto g|_{x}$ are surjective for all $x \in X^{\ast}$. 

\item $S = GX^{\ast}$.

\item Each $\mathscr{L}$-class contains at least one element from $X^{\ast}$.
\end{enumerate}
\end{lemma}
\begin{proof}
(1) $\Rightarrow$ (2).
Suppose that the function from $G$ to itself given by $g \mapsto g|_{x}$ is surjective for all $x \in X^{\ast}$.
Let $xg \in S$.
By assumption there is an element $h \in G$ such that
$h^{-1}|_{x} = g^{-1}$.
Put $y = h^{-1} \cdot x$.
Then
$$hy = h(h^{-1} \cdot x) = [(hh^{-1}) \cdot x]h|_{h^{-1}\cdot x} = x(h^{-1}|_{x})^{-1} = xg$$
using Lemma~\ref{le: two}.
This immediately implies that $S = GX^{\ast}$.

(2) $\Rightarrow$ (3).
Suppose that $S = GX^{\ast}$.
Let $xg \in S$.
Then by assumption, $xg = hy$ for some $h \in G$ and $y \in X^{\ast}$.
Thus $x = h \cdot y$ and $g = h|_{y}$.
It follows by Lemma~\ref{le: greens} that $xg \, \mathscr{L} \, y$.

(3) $\Rightarrow$ (1).
Let $x \in X^{\ast}$ and $g \in G$.
By assumption, $xg^{-1} \, \mathscr{L} \, y$ for some $y \in X^{\ast}$.
By Lemma~\ref{le: greens}, there exists $k \in G$ such that
$k \cdot x = y$ and $k|_{x} = g$.
\end{proof}

We now have the following.

\begin{proposition}\label{prop: symmetric} Let $S$ be a left Rees monoid.
Then the following are equivalent.
\begin{enumerate}

\item The functions from $G$ to itself given by $g \mapsto g|_{x}$ are bijective for all $x \in X^{\ast}$.

\item $S$ is a cancellative monoid in which each $\mathscr{L}$-class contains exactly one element of $X^{\ast}$.

\item $S = GX^{\ast}$ uniquely.

\end{enumerate}
\end{proposition}
\begin{proof}
The equivalence of (1) and (2) follows from Lemmas~\ref{le: yan} and \ref{le: tan}.

(1) and (2) $\Rightarrow$ (3). 
By (1), the functions are in particular surjective and so
by Lemma~\ref{le: tan}, we have that $S = GX^{\ast}$.
Suppose that $gx = hy$ where $g,h \in G$ and $x,y \in X^{\ast}$.
Then $g \cdot x = h \cdot y$ and $g|_{x} = h|_{y}$.
Observe that
$$(h^{-1}g)|_{x} 
= (h^{-1}|_{g \cdot x})g|_{x} 
= (h^{-1}|_{h \cdot y})g|_{x}
= (h|_{y})^{-1}g|_{x} = 1$$
using Lemma~\ref{le: two}.
Also $(h^{-1}g) \cdot x = y$.
It follows by Lemma~\ref{le: greens} that $x\, \mathscr{L} \, y$.
But then by (2), it follows that $x = y$.
We then deduce that $g = h$ using (1), the fact that the functions are injective.
We have shown that every element of $G$ can be written uniquely as a product of an element of $G$ and an element of $X^{\ast}$.
 
(3) $\Rightarrow$ (1).
Suppose that
$S = GX^{\ast}$ {\em uniquely}.
In particular, $S = GX^{\ast}$ and so by Lemma~\ref{le: tan} the function from $G$ to itself given by $g \mapsto g|_{x}$ is surjective.
We now show that these functions are injective.
By part (6) of Lemma~\ref{le: two}, it is enough to check when $x \in X$.
We will verify that the condition of part (7) of Lemma~\ref{le: two} holds.
Suppose that $g|_{x} = 1$.
Then
$gx = (g \cdot x)g|_{x} = g \cdot x$.
Thus
$$1(g \cdot x) = gx.$$
Thus by uniqueness $g = 1$, and so by part (7) of Lemma~\ref{le: two} the function from $G$ to itself
given by $g \mapsto g|_{x}$ is injective.
Thus (1) holds.
\end{proof}

We shall say that a left Rees monoid $S = X^{\ast}G$ is {\em left symmetric} if it is 
a Rees monoid in which each $\mathscr{L}$-class contains exactly one element from $X^{\ast}$.
In particular, this means that the poset of principal right ideals is isomorphic to the poset of principal left ideals.
In a left symmetric Rees monoid, each left basis is a right basis.
A {\em symmetric Rees monoid} is a Rees monoid that is both left and right symmetric.

\section{Decompositions of left Rees monoids}

The goal of this section is to show how each left Rees monoid can be decomposed using the monoid analogues
of free products with amalgamation and HNN extensions.
Free products with amalgamation are a well-known tool in semigroup theory and we use standard results.
However, our definition of HNN extensions for the kinds of monoids that we shall be dealing with appears to be new
and we shall need to develop their theory from scratch.

Let $S = X^{\ast}G$.
We denote the set of orbits of $G$ on $X$ by $\{X_{1}, \ldots, X_{m}\}$.
By Proposition~\ref{prop: level_transitivity}, the number of such orbits is equal to
the number of maximal proper principal ideals of $S$ and so is an intrinsic property of the monoid.
Recall that if the action of $G$ on $X$ is transitive then we say that the associated left Rees monoid is {\em irreducible}.
For each $i \leq i \leq m$ define $S_{i} = X_{i}^{\ast}G$.

\begin{lemma} 
With the above notation, $S_{i}$ is a submonoid of $S$ and is itself an irreducible left Rees monoid.
In addition, the group of units of all the $S_{i}$ is equal to $G$.
If $i \neq j$ then $S_{i} \cap S_{j} = G$.
\end{lemma}
\begin{proof} Let $xg,yh \in X_{i}^{\ast}G$.
Then $(xg)(yh) = x(gy)h$.
By induction, it is easy to prove that $gy = y'g'$ where $y' \in X_{i}^{\ast}$.
Thus $S_{i}$ is closed under multiplication.
The remainder of the lemma is easy to prove.
\end{proof}

We call the left Rees monoids $S_{i}$ constructed above the {\em irreducible components} of the left Rees monoid $S$.
For the next part, we shall recall a classical result proved by Bourbaki \cite{Bourbaki}, 
but we also refer the reader to \cite{Lallement1,Dekov} for alternative proofs.
We state the result in a form dual to that given in 
Proposition~5 of Section~7, p. I. 81 of \cite{Bourbaki}.

\begin{theorem}[Bourbaki]\label{the: bourbaki} Let $S_{i}$, where $i \in I$, be a family of monoids.
We suppose that $G$ is a submonoid of all the $S_{i}$ and that $S_{i} \cap S_{j} = G$ for all $i \neq j$.
Denote the identity of $G$ by $1$.
Assume that for each $i$ there is a subset $F_{i} \subseteq S_{i}$, containing 1, such that
the map from $F_{i} \times G$ to $S_{i}$ given by $(f_{i},g) \mapsto f_{i}a$, with the product taken in $S_{i}$, is a bijection. 
Then every element of $\ast_{G} S_{i}$, the free product with amalgamation of the $S_{i}$, can 
be written uniquely in the form
$$s_{1} \ldots s_{m}g 
\mbox{ where } g \in G \mbox{ and }
s_{1} \in F_{i_{1}}, \ldots, s_{m} \in F_{i_{m}}
\mbox{ and }
i_{1} \neq i_{2},
i_{2} \neq i_{3}, \ldots.$$
\end{theorem}

We now have the following.

\begin{theorem}\label{the: first_decomposition}\mbox{}
\begin{enumerate}

\item Let $S_{1}, \ldots, S_{m}$ be any set of irreducible left Rees monoids whose groups of units are all equal to $G$.
Suppose in addition that $S_{i} \cap S_{j} = G$ whenever $i \neq j$.
Then the free product with amalgamation $S_{1} \, \ast_{G} \, S_{2} \, \ast_{G} \ldots \ast_{G} \, S_{m}$
is a left Rees monoid with group of units $G$ and irreducible components isomorphic to the $S_{i}$.

\item Each left Rees monoid is either irreducible or a free product with amalgamation of irreducible left Rees monoids
having the same groups of units.

\end{enumerate}
\end{theorem}
\begin{proof} 

(1) This follows almost immediately from Theorem~\ref{the: bourbaki} and the structure of left Rees monoids.

(2) Let $S = X^{\ast}G$ be a left Rees monoid.
As above, denote the set of orbits of $G$ on $X$ by $\{X_{1}, \ldots, X_{m}\}$.
For each $i \leq i \leq m$ define $S_{i} = X_{i}^{\ast}G$.
We have proved that each $S_{i}$ is irreducible.
We need to prove that $S$ is isomorphic to $\ast_{G}S_{i}$ but this is virtually immediate by Theorem~\ref{the: bourbaki}.
\end{proof}

Because of the above theorem, we shall focus our attention for the remainder of this section on irreducible left Rees monoids.
The key notion we shall need is a notion of conjugacy for monoids.
Unlike in the group case, our notion of conjugacy will come with a built-in parity.

Let $S$ be a monoid.
An ordered pair of elements $(a,b)$ of $S$ will be said to be {\em right conjugate} with respect
to the element $x \in S$ if $ax = xb$.
We see from Proposition~1.3.4 of \cite{Lothaire}, that the usual notion of conjugacy in free
monoids is related to our notion above.
To work with right conjugacy, we need extra structure.
The proof of the following lemma is trivial. 

\begin{lemma} Let $S$ be a left cancellative monoid and let $x \in S$ be chosen.
\begin{enumerate}

\item Given $a \in S$ there exists at most one element $b \in S$ such that $ax = xb$.

\item Suppose that $a_{1}x = xb_{1}$ and $a_{2}x = xb_{2}$.
Then $a_{1}a_{2}x = xb_{1}b_{2}$.

\end{enumerate}
\end{lemma}

The above lemma enables us to prove the following.
Let $S$ be a left cancellative monoid and let $x \in S$ be fixed.
Define
$$S_{x}^{+} = \{ a \in S \colon ax \in xS\}
\mbox{ and }
S_{x}^{-} = \{ b \in S \colon xb \in aS\}.$$
Both of these are submonoids of $S$ and the map $\phi_{x} \colon S_{x}^{+} \rightarrow S_{x}^{-}$
given by $\phi_{x}(a) = b$ iff $ax = xb$ is a surjective homomorphism.
We now turn this around.\\

\noindent
{\bf Embedding problem }
Let $U$ and $V$ be left cancellative monoids and let $\phi \colon \: U \rightarrow V$ be a surjective homomorphism.
Find a left cancellative monoid $S$ that contains
$U$ and $V$ as submonoids 
and an element $x \in S$
such that
for all $u \in U$ we have that $ux = x \phi(u)$.
We say that $(S,x)$ {\em solves} the embedding problem.\\

It is clear that our question should be viewed as being an extension of the usual theory of HNN extensions for groups
to left cancellative monoids.
We shall not attempt to resolve this question in general.
Instead, we shall consider a special case.\\

\noindent
{\bf Special embedding problem }
Let $U$ and $V$ be {\em groups} and let  $\phi \colon U \rightarrow V$
be a surjective homomorphism.
Find a left cancellative monoid $S$ that contains
$U$ and $V$ as submonoids of the group of units of $S$ and an element $x \in S$ such that
for all $u \in U$ we have that $ux = x \phi(u)$.
We say that $(S,x)$ {\em solves} the special embedding problem.\\

We emphasise that even in the case where $\phi$ is injective, the special embedding problem
requires us to find a left cancellative monoid; in other words, we view groups as being left cancellative
monoids. Of course, we may also regard groups as groups: we shall take up that part of the story in the next section.

We show first that the special embedding problem is directly related to the structure of irreducible left Rees monoids.
Let $S = X^{\ast}G$ be an irreducible left Rees monoid. 
Let $x \in X$. 
There is a surjective homomorphism $\phi_{x} \colon G_{x} \rightarrow \mbox{im}(\phi_{x})$
between two subgroups of $G$.
Now for $h \in G_{x}$ we have that
$$hx = (h \cdot x)h|_{x} = x \phi_{x}(h).$$
Thus $(S,x)$ is a solution to the special emebdding problem for  $\phi_{x} \colon G_{x} \rightarrow \mbox{im}(\phi_{x})$.
It follows that irreducible left Rees monoids are solutions to instances of the special embedding problem.
In fact, the special embedding problem may always be solved by an irreducible left Rees monoid;
this follows immediately from our work in Section~2.3 on constructing irreducible left Rees monoids from
partial endomorphisms of their groups of units.

\begin{proposition} Let $H$ be a group and let $\phi$ be a partial endomorphism of $H$
with domain of definition $U$ and image $V$. 
Then there exists a left Rees monoid $S$ and an element $x \in S$ such that
$(S,x)$ solves the special embedding problem for $\phi$.
\end{proposition}

We now prove that irreducible left Rees monoids are, in some sense, the most general solutions to the
special embedding problem.
First we need a simple lemma.

\begin{lemma}\label{le: normal_form} Let $S = X^{\ast}G$ be an irreducible left Rees monoid where $X = \{x_{i} \colon i \in I\}$.
We assume that $1 \in I$ and put $x = x_{1}$.
Put $H = G_{x}$.
Choose a coset decomposition $G = \bigcup_{i \in I}g_{i}H$. 
Then each element of $S$ can be written uniquely
in the form
$$(k_{1}x) \ldots (k_{p}x)g$$
where $k_{j} \in \{g_{i} \colon i \in I\}$ and $g \in G$ is arbitrary.
\end{lemma}
\begin{proof} We need only prove that the set $\{g_{i}x \colon i \in I\}$ is a 
transversal of the $\mathscr{R}$-classes of the set of irreducible elements
which is straightforward.
\end{proof}

We  showed in Section~2,  how left Rees monoids could be constructed from partial endomorphisms of a group by 
constructing first a covering bimodule and then the tensor monoid of that bimodule.
We now show how left Rees monoids may be constructed directly from that partial endomorphism.

\begin{theorem}\label{the: second_structure_theorem}
Let $V$ be a group and let $\phi$ be a partial endomorphism of $V$ with domain of definition $U$.
Let
$$M = \mbox{\rm Mon}\langle V,x \colon ux = x\phi (u) \mbox{ for all } u \in U \rangle$$
be the monoid with generators $V \cup \{x\}$ and relations all those that hold in $V$ together with
the relations explicitly stated.
Then $M$ is a left Rees monoid isomorphic to the one constructed from the partial endomorphism $\phi$ by the methods of Section~2.
\end{theorem}
\begin{proof} Let $V = \bigcup_{i \in I} g_{i}U$ be a coset decomposition.
An element of $M$ will be a represented by a string over the alphabet $V \cup \{x\}$
where we may multiply together in $V$ consecutive occurrences of elements of $V$.
It follows that an arbitrary element either contains no occurrence of $x$ and so is just an element of $V$ or has the form
$v_{1}xv_{2}xv_{3} \ldots xv_{m}$.
Given our coset decomposition, each element of $V$ can be written uniquely in the form $g_{i}u_{i}$
for some $u_{i} \in U$.
We claim that each element of $M$ is equal in the presentation to an element which is a product 
of elements of the form $g_{i}x$ followed by an element of $V$.
We call an element written in this way a {\em normal form.}
If the string contains no occurrence of $x$ then we simply have an element of $V$.
Suppose the string begins $vx$ where $v \in V$.
Then we may write $v = g_{i}u$.
Thus $vx = g_{i}ux = (g_{i}x)\phi(u)$.
This process may be continued working left-to-right to conclude with a string of the correct form.
Now let $S$ be the irreducible left Rees monoid contructed from the given partial endomorphism.
The group of units of $S$ is $V$.
Let $y$ be an atom in $S$ whose left stabilizer is $U$.
Using the above coset decomposition and Lemma~\ref{le: normal_form},
we see that $S$ is generated by $V$ and $y$ and that $uy = y\phi (u)$  for all $u \in U$. 
It follows that $S$ is a homomorphic image of $M$.
However, using Lemma~\ref{le: normal_form}, 
we see that distinct normal forms in $S$ represent distinct elements of $S$.
It follows that the homomorphism from $M$ to $S$ is also injective.
Thus $M$ is isomorphic to $S$.
\end{proof}

\begin{remark}{\em  The above theorem throws a new light on irreducible self-similar group actions.
They are intimately related to the notion of HNN extensions for left cancellative monoids.}
\end{remark}

\section{Universal groups}

For each monoid $S$ there is a group $U(S)$ and a homomorphism $\iota \colon \: S \rightarrow U(S)$
such that for each homomorphism $\phi \colon \: S \rightarrow G$ to a group there is a unique homomorphism
$\bar{\phi} \colon \: U(S) \rightarrow G$ such that
$\phi = \bar{\phi} \iota$.
The group $U(S)$ is called the {\em universal group} of $S$.
The monoid $S$ can be embedded in a group iff $\iota$ is injective.

\begin{remark} {\em In \cite{Cohn} and independently in the thesis of von Karger \cite{V},
it is proved, in completely different ways, that right rigid cancellative monoids can always be
embedded in their universal groups.
This result motivated this section but we shall not assume their results here.}
\end{remark}

We begin with the irreducible case.

\begin{proposition}\label{prop: universal1} Let $S = X^{\ast}G$ be an irreducible Rees monoid with group of units $G$.
Let an associated partial automorphism of $G$  
be $\phi \colon A \rightarrow B$.
Then the universal group of $S$ is isomorphic to the group HNN extension
$$G^{\ast} = \mbox{\rm Grp}\langle G,t \colon at = t\phi (a) \mbox{ for all } a \in A \rangle$$
and $S$ can be embedded in $G^{\ast}$.
\end{proposition}
\begin{proof} 
We have proved that $S$ is isomorphic to the monoid
$$\mbox{\rm Mon}\langle G,t \colon at = t\phi (a) \mbox{ for all } a \in A \rangle.$$
Choose a coset decomposition $G = \bigcup_{i \in I} g_{i}A$ where $g_{1} = 1$.
We have proved that each element of $S$ can be written uniquely in the form
$$(k_{1}x) \ldots (k_{p}x)g$$
where $k_{l} \in \{g_{i} \colon i \in I\}$ and $g \in G$.
Choose also a coset decomposition $G = \bigcup_{j \in J} h_{j}B$ where $h_{1} = 1$.
By Theorem~IV~2.1 of \cite{LS} and reversing the order in which we write normal forms,
each element of $G^{\ast}$ may be written uniquely as a product of elements of the form
$g't$ where $g' \in  \{g_{i} \colon i \in I\}$ 
and
$h't^{-1}$ where $h' \in  \{h_{j} \colon j \in J\}$ 
followed by an arbitrary element of $G$ under the further condition
that neither $t1t^{-1}$ nor $t^{-1}1t$ occur.
It follows immediately from these normal forms that $S$ may be embedded in $G^{\ast}$.

It remains to show that $G^{\ast}$ is the universal group of $S$.
We shall prove this from the normal forms.
Let $\theta \colon \: S \rightarrow U$ be any homomorphism from $S$ to a group $U$.
From the normal form for elements of $G^{\ast}$ there is a unique extension 
of the function $\theta$ to a function $\psi \colon \: G^{\ast} \rightarrow U$
determined by the fact that $\psi (t) = \theta (t)$ and $\psi$ agrees with $\theta$ on $G$.
It remains to show that $\psi$ is a homomorphism.
When we multiply two normal forms in $G^{\ast}$ together, 
we apply the rules 
$at = t \phi (a)$, whenever $a \in A$, 
and
$bt^{-1} = t^{-1} \phi^{-1}(b)$, whenever $b \in B$,
and 
$tt^{-1} = t^{-1}t = 1$.
But  $\theta (a) \theta (t) = \theta (t) \theta (\phi (a))$ 
and
$\theta (b) \theta (t)^{-1} = \theta (t)^{-1} \theta (\phi^{-1}(b))$
and the result follows.
\end{proof}

The following is now immediate.

\begin{corollary} An irreducible left Rees monoid can be embedded in a group
if and only if it is right cancellative.
\end{corollary}

We may now generalize the above results.

\begin{proposition}\label{prop: universal2} 
Let $S$ be a  Rees monoid with group of units $G$.
Let the irreducible components of $S$ be $S_{1}, \ldots, S_{m}$.
Then the universal group of $S$ is the amalgamated free product
$U(S_{1}) \ast_{G} U(S_{2}) \ast_{G} \ldots \ast_{G} U(S_{m})$.
Furthermore, $S$ is embedded in its universal group.
\end{proposition}
\begin{proof} We use the fact that $S$ is isomorphic to the amalgamated free product
$S_{1} \ast_{G} S_{2} \ast_{G} \ldots \ast_{G} S_{m}$.
This means that $S$ is the pushout of the set of maps $G \rightarrow S_{i}$.
Let $\theta \colon \: S \rightarrow V$ be any homomorphism to a group $V$.
This induces homomorphisms from $S_{i} \rightarrow V$ which all agree on $G$.
It follows that we have homomorphisms from each $U(S_{i})$ to $V$
and so by the usual properties of pushouts in the category of groups we have a unique homomorphism
from $U(S_{1}) \ast_{G} U(S_{2}) \ast_{G} \ldots \ast_{G} U(S_{m})$ to $V$.
We have therefore established the claim concerning the universal group.
To show that we have an embedding, we use the fact that each group $U(S_{i})$ is embedded in the amalgamated free product.
Thus by Proposition~\ref{prop: universal1}, we have that each monoid $S_{i}$ is embedded in the free product
via its embedding in the group $U(S_{i})$.
It follows from the structure of amalgamated free products of groups that $S$ is also embedded.
\end{proof}

The following is just a restatement of the above in its general form and, as we explained above,
a direct proof of a special case of a theorem proved in \cite{Cohn,V}.

\begin{theorem}\label{the: embedding} Left Rees monoids can be embedded in groups if and only if
they are right cancellative.
\end{theorem}

\begin{example} {\em Consider the embedding $\phi \colon n\mathbb{Z} \rightarrow \mathbb{Z}$ given by $\phi (n) = m$.
Then we may construct an irreducible Rees monoid $S = X^{\ast} \bowtie \mathbb{Z}$ where $\left| X \right| = n$.
The universal group is $BS(m,n)$, a {\em Baumslag-Solitar group}.
It is therefore natural to refer to the corresponding Rees monoids as {\em Baumslag-Solitar monoids}.
This term has also been used by Alan Cain \cite{Cain}.
His semigroups are subsemigroups of ours of the form $X^{\ast} \bowtie \mathbb{N}$.}
\end{example}

Every right reversible cancellative monoid can be embedded in a group $G$
in such a way that $G = S^{-1}S$.
Such a group is called a {\em group of left quotients} of $S$
and is uniquely determined up to isomorphism.
See Section~1.10 of \cite{CP} for proofs of all these assertions.

\begin{corollary} 
The universal group of a recurrent Rees monoid is a group of left quotients.
\end{corollary}

We conclude with one final case.
Let $S$ be a left symmetric Rees monoid.
This means that the functions $g \mapsto g|_{x}$ are bijective for all $x \in X^{\ast}$.

\begin{proposition} Let $S = X^{\ast}G$ be a left symmetric Rees monoid.
Then the universal group of $S$ is isomorphic to $FG(X) \bowtie G$ where $FG(X)$ is the free group on $X$.
\end{proposition}
\begin{proof} The first step is to prove that we may construct the Zappa-Sz\'ep product $FG(X) \bowtie G$
given that we already have the  Zappa-Sz\'ep product $X^{\ast} \bowtie G$.
We denote the map $g \mapsto g|_{x}$ by $\rho_{x}$ which, by assumption, is a bijection.
We first define the function $G \times FG(X) \rightarrow FG(X)$ given by $(g,w) \mapsto g|_{w}$
as follows.
Let $x \in X$.
Define
$$g |_{x^{-1}} = \rho_{x}^{-1}(g).$$
If $w \in FG(X)$ is the reduced string $w = x_{1}^{\epsilon_{1}} \ldots x_{m}^{\epsilon_{m}}$, where the superscripts are either plus one or minus one,
then define
$$g |_{w} = ( \ldots (g |_{x_{1}^{\epsilon_{1}}}) \ldots) |_{x_{m}^{\epsilon_{m}}}.$$ 
Next we define the function $G \times FG(X) \rightarrow FG(X)$ given by $(g,w) \mapsto g \cdot w$.
Let $x \in X$.
Define
$$g \cdot x^{-1} = (g|_{x^{-1}}  \cdot x)^{-1}.$$
We shall now extend this definition to all reduced strings in $FG(X)$ in such a way
that $g \cdot (uv) = (g \cdot u)(g|_{u} \cdot v)$ for all $u,v \in FG(X)$.
This can be proved by means of induction.
Observe that 
$$g \cdot (xx^{-1}) = 1 g \cdot (x^{-1}x) = 1$$
and
$$g |_{xx^{-1}} = g = g|_{x^{-1}x}$$
however they are computed.
This means that our definitions are well-defined when working with elements of the free group.
One can now check that the axioms (SS1)--(SS8), generalized to this setting as explained in Remark~3.4,
all hold; only (SS8) and (SS2) needing any work.
See Theorem~2.4.1 of \cite{AW} for details.
We have therefore constructed the group  $FG(X) \bowtie G$  into which the monoid $S$ is embedded in the obvious way.
It remains to show that it is also the universal group.

Let $\theta \colon S \rightarrow H$ be a homomorphism to the group $H$.
This restricts to a map from $X$ to $H$ and therefore to a unique homomorphism
$\alpha \colon FG(X) \rightarrow H$.
Denote the restriction of $\theta$ to $G$ by $\beta$.
If $wg \in FG(X) \bowtie G$ define $\phi (wg) = \alpha (w) \beta (g)$.
Clearly, $\phi$ agrees with $\theta$ when restricted to $S$.
If we can proved that $\phi$ is a homomorphism it will necessarily be the unique homomorphism
extending $\theta$.

To prove that $\phi$ is a homomorphism,
it is enough to show that
$$\beta (g) \alpha (x^{-1}) = \alpha (g \cdot x^{-1}) \beta (g|_{x^{-1}}).$$
We have that
$$xg^{-1} = \rho_{x}^{-1}(g)^{-1}(\rho_{x}^{-1}(g) \cdot x).$$
Thus
$$\alpha (x) \beta (g)^{-1} = \beta (\rho_{x}^{-1}(g))^{-1}\alpha (\rho_{x}^{-1}(g) \cdot x).$$
Using the fact that $\alpha (x)^{-1} = \alpha (x^{-1})$, we get that
$$\beta (g) \alpha (x^{-1})
=
\alpha (g \cdot x^{-1})
 \beta (g|_{x^{-1}}),$$
as required.
\end{proof}




\end{document}